\documentclass{amsart}

\usepackage{graphicx}
\usepackage[centertags]{amsmath}
\usepackage{amsfonts}
\usepackage{amssymb}
\usepackage{amsthm}
\usepackage{newlfont}
\usepackage{hyperref}
\usepackage{mathrsfs} 
\usepackage{yfonts} 

\newtheorem*{thm*}{Theorem}
\newtheorem{thm}{Theorem}[section]
\newtheorem{cor}[thm]{Corollary}
\newtheorem{lem}[thm]{Lemma}
\newtheorem{prop}[thm]{Proposition}

\theoremstyle{definition}
\newtheorem{defn}[thm]{Definition}

\theoremstyle{remark}
\newtheorem{rem}[thm]{Remark}

\numberwithin{equation}{section}

\DeclareSymbolFont{bbold}{U}{bbold}{m}{n}
\DeclareSymbolFontAlphabet{\mathbbold}{bbold}

\newcommand{\N}{\mathbb{N}}
\newcommand{\Z}{\mathbb{Z}}
\newcommand{\R}{\mathbb{R}}

\newcommand{\acts}{\curvearrowright}

\newcommand{\pmp}{p{$.$}m{$.$}p{$.$}}

\newcommand{\pv}{\bar{p}}
\newcommand{\qv}{\bar{q}}
\newcommand{\rv}{\bar{r}}

\newcommand{\cF}{\mathcal{F}}
\newcommand{\cJ}{\mathscr{I}}
\newcommand{\cK}{\mathscr{K}}
\newcommand{\cL}{\mathscr{L}}
\newcommand{\cP}{\mathcal{P}}
\newcommand{\cQ}{\mathcal{Q}}
\newcommand{\cR}{\mathcal{R}}

\newcommand{\E}{\mathscr{E}}

\newcommand{\dist}{\mathrm{dist}}
\newcommand{\dom}{\mathrm{dom}}
\newcommand{\rng}{\mathrm{rng}}
\newcommand{\symd}{\triangle}
\newcommand{\salg}{\sigma \text{-}\mathrm{alg}}
\newcommand{\ralg}{\sigma \text{-}\mathrm{alg}^{\mathrm{red}}}
\newcommand{\sH}{\mathrm{H}}
\newcommand{\dB}{d} 
\newcommand{\dR}{d^{\mathrm{Rok}}} 
\newcommand{\Prt}{\mathscr{P}} 
\newcommand{\HPrt}{\mathscr{P}_\sH} 
\newcommand{\rh}{h^{\mathrm{Rok}}}
\newcommand{\suprh}[1]{h^{\mathrm{Rok}}_{sup}(#1)}
\newcommand{\Borel}{\mathcal{B}}
\newcommand{\given}{\mathbin{|}}
\newcommand{\Given}{\Big|}
\renewcommand{\:}{\,:\,}
\newcommand{\res}{\restriction}

\begin{document}

\title[Krieger's finite generator theorem for countable groups II]{Krieger's finite generator theorem for actions of countable groups II}
\author{Brandon Seward}
\address{Department of Mathematics, University of Michigan, 530 Church Street, Ann Arbor, MI 48109, U.S.A.}
\email{b.m.seward@gmail.com}

\begin{abstract}
We continue the study of Rokhlin entropy, an isomorphism invariant for {\pmp} actions of countable groups introduced in the previous paper. We prove that every free ergodic action with finite Rokhlin entropy admits generating partitions which are almost Bernoulli, strengthening the theorem of Ab\'{e}rt--Weiss that all free actions weakly contain Bernoulli shifts. We then use this result to study the Rokhlin entropy of Bernoulli shifts. Under the assumption that every countable group admits a free ergodic action of positive Rokhlin entropy, we prove that: (i) the Rokhlin entropy of a Bernoulli shift is equal to the Shannon entropy of its base; (ii) Bernoulli shifts have completely positive Rokhlin entropy; and (iii) Gottschalk's surjunctivity conjecture and Kaplansky's direct finiteness conjecture are true.
\end{abstract}
\maketitle

\section{Introduction}

Let $(X, \mu)$ be a standard probability space, meaning $X$ is a standard Borel space and $\mu$ is a Borel probability measure. Let $G$ be a countable group and let $G \acts (X, \mu)$ be a probability-measure-preserving (\pmp) action. For a collection $\mathcal{C}$ of Borel subsets of $X$, we let $\salg_G(\mathcal{C})$ denote the smallest $G$-invariant $\sigma$-algebra containing $\mathcal{C} \cup \{X\}$ and the null sets. A Borel partition $\alpha$ is \emph{generating} if $\salg_G(\alpha)$ is the entire Borel $\sigma$-algebra $\Borel(X)$. For finite $T \subseteq G$ we write $\alpha^T$ for the join of the translates $t \cdot \alpha$, $t \in T$, where $t \cdot \alpha = \{t \cdot A : A \in \alpha\}$. The \emph{Shannon entropy} of a countable Borel partition $\alpha$ is
$$\sH(\alpha) = \sum_{A \in \alpha} - \mu(A) \cdot \log(\mu(A)).$$
A \emph{probability vector} is a finite or countable ordered tuple $\pv = (p_i)$ of positive real numbers which sum to $1$. We write $|\pv|$ for the length of $\pv$ and $\sH(\pv) = \sum - p_i \cdot \log(p_i)$ for the Shannon entropy of $\pv$.

In Part I of this series \cite{S14}, we defined the \emph{Rokhlin entropy} of a {\pmp} action $G \acts (X, \mu)$ as
$$\rh_G(X, \mu) = \inf \Big\{ \sH(\alpha \given \cJ) : \alpha \text{ is a countable partition and } \salg_G(\alpha) \vee \cJ = \Borel(X) \Big\},$$
where $\cJ$ is the $\sigma$-algebra of $G$-invariant sets. In this paper, we will only be interested in ergodic actions, in which case the Rokhlin entropy simplifies to
$$\rh_G(X, \mu) = \inf \Big\{ \sH(\alpha) : \alpha \text{ is a countable generating partition} \Big\}.$$
When $G$ is amenable and the action is free, the Rokhlin entropy coincides with classical Kolmogorov--Sinai entropy \cite{ST14,AS}. Rokhlin entropy is thus a natural analog of classical entropy. The main theorem of the prequel was the following generalization of Krieger's finite generator theorem.

\begin{thm}[\cite{S14}] \label{intro:krieger}
Let $G$ be a countably infinite group acting ergodically, but not necessarily freely, by measure-preserving bijections on a non-atomic standard probability space $(X, \mu)$. If $\pv = (p_i)$ is any finite or countable probability vector with $\rh_G(X, \mu) < \sH(\pv)$, then there is a generating partition $\alpha = \{A_i \: 0 \leq i < |\pv|\}$ with $\mu(A_i) = p_i$ for every $0 \leq i < |\pv|$.
\end{thm}

In the present paper we use the above theorem to study the Rokhlin entropy of Bernoulli shifts. Recall that for a standard probability space $(L, \lambda)$ the \emph{Bernoulli shift} over $G$ with \emph{base space} $(L, \lambda)$ is simply the product space $(L^G, \lambda^G)$ equipped with the natural left-shift action of $G$:
$$\text{for } g, h \in G \text{ and } x \in L^G \quad (g \cdot x)(h) = x(g^{-1} h).$$
The \emph{Shannon entropy} of the base space is
$$\sH(L, \lambda) = \sum_{\ell \in L} - \lambda(\ell) \cdot \log \lambda(\ell)$$
if $\lambda$ has countable support, and $\sH(L, \lambda) = \infty$ otherwise. Every Bernoulli shift $(L^G, \lambda^G)$ comes with the canonical, possibly uncountable, generating partition $\cL = \{R_\ell : \ell \in L\}$, where
$$R_\ell = \{x \in L^G : x(1_G) = \ell\}.$$
Note that if $\sH(L, \lambda) < \infty$ then $\cL$ is countable and $\sH(\cL) = \sH(L, \lambda)$. Thus one always has $\rh_G(L^G, \lambda^G) \leq \sH(L, \lambda)$.

A fundamental open problem in ergodic theory is to determine, for every countably infinite group $G$, whether $(2^G, u_2^G)$ can be isomorphic to $(3^G, u_3^G)$. Here we write $n$ for $\{0, \ldots, n-1\}$ and $u_n$ for the normalized counting measure on $\{0, \ldots, n-1\}$. Note that $\sH(n, u_n) = \log(n)$. For amenable groups $G$, the Bernoulli shift $(L^G, \lambda^G)$ has Kolmogorov--Sinai entropy $\sH(L, \lambda)$, and thus $(2^G, u_2^G)$ and $(3^G, u_3^G)$ are non-isomorphic. In 2010, groundbreaking work of Bowen \cite{B10b}, together with improvements by Kerr and Li \cite{KL11a}, created the notion of sofic entropy for {\pmp} actions of sofic groups. We remind the reader that the class of sofic groups contains the countable amenable groups, and it is an open question whether every countable group is sofic. Sofic entropy extends Kolmogorov--Sinai entropy, as when the acting sofic group is amenable the two notions coincide \cite{B12, KL13}. For sofic $G$, the Bernoulli shift $(L^G, \lambda^G)$ has sofic entropy $\sH(L, \lambda)$ \cite{B10b, KL11b}. Thus $(2^G, u_2^G)$ and $(3^G, u_3^G)$ are non-isomorphic for sofic $G$. Based on these results, it seems that the following statement may be true of all countably infinite groups $G$:
$$\mathbf{INV:} \ \sH(L, \lambda) \text{ is an isomorphism invariant for } (L^G, \lambda^G).$$

\begin{rem}
Another important question is whether $\sH(L, \lambda) = \sH(K, \kappa)$ implies that $(L^G, \lambda^G)$ is isomorphic to $(K^G, \kappa^G)$. In 1970, Ornstein famously answered this question positively for $G = \Z$, thus completely classifying Bernoulli shifts over $\Z$ up to isomorphism \cite{Or70a, Or70b}. This result was extended to amenable groups by Ornstein and Weiss in 1987 \cite{OW87}. Work of Stepin shows that this property is retained under passage to supergroups \cite{St75}, so the isomorphism result extends to all groups which contain an infinite amenable subgroup. In 2012, Bowen proved that for every countably infinite group $G$, if $\sH(L, \lambda) = \sH(K, \kappa)$ and the supports of $\lambda$ and $\kappa$ each have cardinality at least $3$, then $(L^G, \lambda^G)$ is isomorphic to $(K^G, \kappa^G)$ \cite{B12b}. Thus, this question is nearly resolved with only the case of a two atom base space incomplete.
\end{rem}

We previously noted that one always has $\rh_G(L^G, \lambda^G) \leq \sH(L, \lambda)$. When $G$ is sofic, Rokhlin entropy is bounded below by sofic entropy \cite{B10b,AS} and thus $\rh_G(L^G, \lambda^G) = \sH(L, \lambda)$ whenever $G$ is sofic. Since the definition of Rokhlin entropy does not require the acting group to be sofic, the statement
$$\mathbf{RBS:} \ \rh_G(L^G, \lambda^G) = \sH(L, \lambda) \text{ for every standard probability space } (L, \lambda).$$
(acronym for Rokhlin entropy of Bernoulli Shifts) may be true for all countably infinite groups $G$. Notice that {\bf RBS} $\Rightarrow$ {\bf INV}.

In this paper we investigate {\bf RBS} and its consequences. We first show in Section \ref{sec:trans} that Rokhlin entropy cannot be realized by a generating partition whose translates are correlated. 

\begin{thm} \label{intro:drop}
Let $G$ be a countably infinite group, let $G \acts (X, \mu)$ be a free {\pmp} ergodic action, and let $\alpha$ be a countable generating partition. If $T \subseteq G$ is finite, $\epsilon > 0$, and $\frac{1}{|T|} \cdot \sH(\alpha^T) < \sH(\alpha) - \epsilon$, then $\rh_G(X, \mu) < \sH(\alpha) - \epsilon / (16 |T|^3)$.
\end{thm}

\begin{rem}
This result was later improved to $\rh_G(X, \mu) \leq \inf_T \frac{1}{|T|} \cdot \sH(\alpha^T)$, where the infimum is over all finite $T \subseteq G$ \cite{S16}.
\end{rem}

When $\sH(\alpha) < \infty$, the equality $\sH(\alpha^T) = |T| \cdot \sH(\alpha)$ implies that the $T$-translates of $\alpha$ are mutually independent. So we obtain the following.

\begin{cor} \label{intro:realize}
Let $G$ be a countably infinite group acting freely and ergodically on a standard probability space $(X, \mu)$ by measure-preserving bijections. If $\alpha$ is a countable generating partition and
$$\rh_G(X, \mu) = \sH(\alpha) < \infty,$$
then $(X, \mu)$ is isomorphic to a Bernoulli shift.
\end{cor}

As sofic entropy is always bounded above by Rokhlin entropy \cite{B10b,AS}, we have the following immediate corollary.

\begin{cor} \label{intro:sofic}
Let $G$ be a sofic group with sofic approximation $\Sigma$, and let $G$ act freely and ergodically on a standard probability space $(X, \mu)$ by measure-preserving bijections. If $\alpha$ is a countable generating partition and the sofic entropy $h_G^\Sigma(X, \mu)$ satisfies $h_G^\Sigma(X, \mu) = \sH(\alpha) < \infty$, then $(X, \mu)$ is isomorphic to a Bernoulli shift.
\end{cor}

With the above corollary we answer a question of N.-P. Chung in \cite[Question 5.4]{Ch13} regarding equilibrium states for sofic pressure for a certain class of functions. See Corollary \ref{cor:pressure}.

From Theorem \ref{intro:drop} we derive in Section \ref{sec:gott} a few properties which would follow if {\bf RBS} were found to be true. Recall that an action $G \acts (X, \mu)$ of an amenable group $G$ is said to have \emph{completely positive entropy} if every factor $G \acts (Y, \nu)$ of $(X, \mu)$, with $Y$ not essentially a single point, has positive Kolmogorov--Sinai entropy. For $G = \Z$, these actions are also called Kolmogorov or K-automorphisms. The standard example of completely positive entropy actions are Bernoulli shifts (see \cite{RW00}). In fact, for amenable groups factors of Bernoulli shifts are Bernoulli \cite{OW87}, but it is unknown if this holds for any non-amenable group. Recently, it was proven by Kerr that Bernoulli shifts over sofic groups have completely positive sofic entropy \cite{Ke13a}. Along these lines, we obtain the following corollary of Theorem \ref{intro:drop}.

\begin{cor} \label{intro:cpre}
Let $G$ be a countably infinite group. Assume that $\rh_G(L^G, \lambda^G) = \sH(L, \lambda)$ for all standard probability spaces $(L, \lambda)$. Then every Bernoulli shift over $G$ has completely positive Rokhlin entropy.
\end{cor}

Our next corollary relates to two well-known open conjectures from outside ergodic theory. The first is \emph{Kaplansky's direct finiteness conjecture}, which states that for every countable group $G$ and every field $K$, if $a$ and $b$ are elements of the group ring $K[G]$ and satisfy $a b = 1$ then $b a = 1$. Kaplansky proved this for $K = \mathbb{C}$ in 1972 \cite{Ka72} (see also a shorter proof by Burger and Valette \cite{BV98}). For general fields $K$, this conjecture was proven for abelian groups by Ara, O'Meara, and Perera in 2002 \cite{AOP02}, and then proven for sofic groups by Elek and Szab\'{o} in 2004 \cite{ES04}.

The second conjecture is \emph{Gottschalk's surjunctivity conjecture}, which states that if $G$ is a countable group, $n \in \N$, and $\phi : n^G \rightarrow n^G$ is a continuous $G$-equivariant injection, then $\phi$ is surjective. This conjecture has a simple topological proof when $G$ is residually finite (this is due to Lawton, see \cite{Got} or \cite{W00}), and can be proven for amenable groups using topological entropy. Gromov proved the conjecture for sofic groups, and in fact he defined the class of sofic groups for this purpose \cite{Gr99, W00}. Later, after the discovery of sofic entropy, a topological entropy proof was given for sofic groups \cite{KL11a}. We point out that it is known that Gottschalk's surjunctivity conjecture implies Kaplansky's direct finiteness conjecture \cite[Section I.5]{CL13}.

From Corollary \ref{intro:realize} we deduce the following.

\begin{cor} \label{intro:gott}
Let $G$ be a countably infinite group. Assume that $\rh_G(L^G, \lambda^G) = \sH(L, \lambda)$ for all standard probability spaces $(L, \lambda)$. Then $G$ satisfies Gottschalk's surjunctivity conjecture and Kaplansky's direct finiteness conjecture.
\end{cor}

If we define the statements
\begin{align*}
\mathbf{CPE:} & \text{ Every Bernoulli shift over } G \text{ has completely positive Rokhlin entropy}.\\
\mathbf{GOT:} & \ G \text{ satisfies Gottschalk's surjunctivity conjecture}.\\
\mathbf{KAP:} & \ G \text{ satisfies Kaplansky's direct finiteness conjecture}.
\end{align*}
then from earlier comments and Corollaries \ref{intro:cpre} and \ref{intro:gott} we deduce that for every countably infinite group $G$
$$\mathbf{RBS} \ \Rightarrow \ \mathbf{INV} + \mathbf{CPE} + \mathbf{GOT} + \mathbf{KAP}.$$

Beginning with Section \ref{sec:approx}, the rest of the paper investigates the validity of {\bf RBS}. We remark that, a priori, there is nothing obvious one can say about $\rh_G(L^G, \lambda^G)$ except that
$$\rh_G((L \times K)^G, (\lambda \times \kappa)^G) \leq \rh_G(L^G, \lambda^G) + \rh_G(K^G, \kappa^G) \leq \sH(L, \lambda) + \sH(K, \kappa).$$
Indeed, we do not know if Rokhlin entropy is additive under direct products, even for Bernoulli shifts.

For a countably infinite group $G$, define
$$\suprh{G} = \sup_{G \acts (X, \mu)} \rh_G(X, \mu),$$
where the supremum is taken over all free ergodic {\pmp} actions $G \acts (X, \mu)$ with $\rh_G(X, \mu) < \infty$. We will relate the validity of {\bf RBS} to the following two statements.
\begin{align*}
\mathbf{POS:} & \text{ There is a free ergodic {\pmp} action } G \acts (X, \mu) \text{ with } \rh_G(X, \mu) > 0.\\
\mathbf{INF:} & \ \suprh{G} = \infty.
\end{align*}
Both statements are known to be true when $G$ is a countably infinite sofic group since sofic entropy is a lower bound to Rokhlin entropy. We do not know whether {\bf POS} implies {\bf INF} (see the discussion following Corollary \ref{cor:infjump}).

Our main tool to study {\bf RBS} is the construction, in Section \ref{sec:approx}, of generating partitions $\alpha$ which are almost Bernoulli in the sense that $\sH(\alpha^T) / |T| > \sH(\alpha) - \epsilon$ for some large but finite $T \subseteq G$ and some small $\epsilon > 0$. By well known properties of Shannon entropy \cite[Fact 3.1.3]{Do11}, this condition is equivalent to saying that the $T$-translates of $\alpha$ are close to being mutually independent. The theorem below may be viewed as a generalization of a similar result obtained by Grillenberger and Krengel for $G = \Z$ \cite{GK76}.

\begin{thm} \label{intro:bkrieger}
Let $G$ be a countably infinite group acting freely and ergodically on a standard probability space $(X, \mu)$ by measure-preserving bijections. If $\pv = (p_i)$ is any finite or countable probability vector with $\rh_G(X, \mu) < \sH(\pv) < \infty$, then for every finite $T \subseteq G$ and $\epsilon > 0$ there is a generating partition $\alpha = \{A_i \: 0 \leq i < |\pv|\}$ with $\mu(A_i) = p_i$ for every $0 \leq i < |\pv|$ and
$$\frac{1}{|T|} \cdot \sH(\alpha^T) > \sH(\alpha) - \epsilon.$$
\end{thm}

Note that $\sH(\alpha)$ could be tremendously larger than $\rh_G(X, \mu)$.

The above theorem strengthens the result of Ab\'{e}rt and Weiss that all free actions weakly contain a Bernoulli shift \cite{AW13}. Specifically, assuming only that $\sH(\pv) > 0$, they proved the existence of an $\alpha$ which is not necessarily generating but otherwise satisfies the conditions stated in the above theorem.

In Section \ref{sec:fin} we establish two semi-continuity properties of Rokhlin entropy, and then we use these semi-continuity properties and Theorem \ref{intro:bkrieger} in order to prove the theorem below. This theorem addresses the validity of {\bf RBS} when $\sH(L, \lambda) < \infty$.

\begin{thm} \label{intro:finite}
Let $G$ be a countably infinite group and let $(L, \lambda)$ be a standard probability space with $\sH(L, \lambda) < \infty$. Then
$$\rh_G(L^G, \lambda^G) = \min \Big( \sH(L, \lambda), \ \suprh{G} \Big).$$
\end{thm}

Note that when $\rh_G(L^G, \lambda^G) < \sH(L, \lambda)$, the supremum $\suprh{G}$ is achieved by $(L^G, \lambda^G)$. We point out that the above theorem places a significant restriction on the nature of the map $\sH(L, \lambda) \mapsto \rh_G(L^G, \lambda^G)$. Prior to obtaining this theorem, there is no obvious reason why this map should be monotone or even piece-wise linear.

In Section \ref{sec:fin} we also prove the following.

\begin{thm} \label{intro:kill}
Let $P$ be a countable group containing arbitrarily large finite subgroups. If $G$ is any countably infinite group with $\suprh{G} < \infty$ then $\suprh{P \times G} = 0$.
\end{thm}

Thus $(\forall G \ \mathbf{POS}) \Rightarrow (\forall G \ \mathbf{INF})$.

In Section \ref{sec:inf} we first establish a formula for the Rokhlin entropy of inverse limits of actions (Theorem \ref{thm:ks}), and then we use this formula together with Theorem \ref{intro:finite} in order to study {\bf RBS} when $\sH(L, \lambda) = \infty$. In the case $\sH(L, \lambda) = \infty$ we obtain a result stronger than Theorem \ref{intro:finite}. This is surprising from a historical perspective, since when Kolmogorov defined entropy in 1958 he could only handle Bernoulli shifts with a finite Shannon entropy base \cite{Ko58,Ko59}. It was not until the improvements of Sinai that infinite Shannon entropy bases could be considered \cite{Si59}. Similarly, when Bowen defined sofic entropy he studied Bernoulli shifts with both finite and infinite Shannon entropy bases \cite{B10b}, but he was only fully successful in the finite case. The infinite case was resolved through improvements by Kerr and Li \cite{KL11a,KL11b,Ke13}.

\begin{thm} \label{intro:infinite}
Let $G$ be a countably infinite group and let $(L, \lambda)$ be a standard probability space with $\sH(L, \lambda) = \infty$. Then $\rh_G(L^G, \lambda^G) = \infty$ if and only if there exists a free ergodic {\pmp} action $G \acts (X, \mu)$ with $\rh_G(X, \mu) > 0$.
\end{thm}

Thus, if $\sH(L, \lambda) = \infty$ then $\rh_G(L^G, \lambda^G)$ is either $0$ or infinity.

It follows from Theorems \ref{intro:finite} and \ref{intro:infinite} that for every countably infinite group $G$
$$\mathbf{INF} \Rightarrow \mathbf{RBS}.$$

By putting all of our results together, we obtain the following.

\begin{cor} \label{intro:posall}
Assume that every countably infinite group $G$ admits a free ergodic {\pmp} action with $\rh_G(X, \mu) > 0$. Then:
\begin{enumerate}
\item[\rm (i)] $\rh_G(L^G, \lambda^G) = \sH(L, \lambda)$ for every countably infinite group $G$ and every probability space $(L, \lambda)$;
\item[\rm (ii)] Every Bernoulli shift over any countably infinite group has completely positive Rokhlin entropy;
\item[\rm (iii)] Gottschalk's surjunctivity conjecture is true;
\item[\rm (iv)] Kaplansky's direct finiteness conjecture is true.
\end{enumerate}
\end{cor}

This corollary indicates that the validity of $(\forall G \ \mathbf{POS})$ should be considered an important open problem.

Finally, for convenience to the reader we summarize the implications we uncovered in the two lines below:
$$\mathbf{INF} \Rightarrow \mathbf{RBS} \Rightarrow \mathbf{INV} + \mathbf{CPE} + \mathbf{GOT} + \mathbf{KAP}$$
$$(\forall G \ \mathbf{POS}) \Rightarrow (\forall G \ \mathbf{INF}).$$ 

\subsection*{Acknowledgments}
This research was partially supported by the National Science Foundation Graduate Student Research Fellowship under Grant No. DGE 0718128. The author thanks his advisor, Ralf Spatzier, for numerous productive conversations, Tim Austin for many suggestions to improve the paper, and Damien Gaboriau for helpful discussions. Finally, the author thanks Lewis Bowen for pointing out that Corollary \ref{intro:sofic} provides an answer to the question \cite[Question 5.4]{Ch13} asked by N.-P. Chung.

\section{Preliminaries} \label{sec:prelim}

Throughout this paper, whenever working with a probability space $(X, \mu)$ we will generally ignore sets of measure zero. In particular, we write $A = B$ for $A, B \subseteq X$ if their symmetric difference is null. Similarly, we will use the term \emph{probability vector} more freely than described in the introduction. A probability vector $\pv = (p_i)$ will be any finite or countable ordered tuple of non-negative real numbers which sum to $1$ (so some terms $p_i$ may be $0$).

Every probability space $(X, \mu)$ which we consider will be assumed to be standard. In particular, $X$ will be a standard Borel space. A well-known property of standard Borel spaces is that they are \emph{countably generated} \cite[Prop. 12.1]{K95}, meaning there is a sequence $B_n \subseteq X$ of Borel sets such that $\Borel(X)$ is the smallest $\sigma$-algebra containing all of the sets $B_n$. In particular, every sub-$\sigma$-algebra $\cF$ is countably generated mod $\mu$-null sets, since the factor $(Y, \nu)$ of $(X, \mu)$ associated to $\cF$ is standard. For $\mathcal{C} \subseteq \Borel(X)$, we let $\salg(\mathcal{C})$ denote the smallest sub-$\sigma$-algebra containing $\mathcal{C} \cup \{X\}$ and the $\mu$-null sets (not to be confused with the notation $\salg_G(\mathcal{C})$ from the introduction). When $G \acts X$ is a Borel action, we write $\E_G(X)$ for the collection of ergodic invariant Borel probability measures on $X$.

For a countable ordered partition $\alpha = \{A_i : 0 \leq i < |\alpha|\}$ we let $\dist(\alpha)$ denote the probability vector $\pv$ satisfying $p_i = \mu(A_i)$. For two partitions $\alpha$ and $\beta$ we write $\alpha \geq \beta$ if $\alpha$ is finer than $\beta$. We let $\HPrt$ denote the set of countable Borel partitions $\alpha$ with $\sH(\alpha) < \infty$. The space $\HPrt$ is a complete separable metric space \cite[Fact 1.7.15]{Do11} under the \emph{Rokhlin metric} $\dR_\mu$ defined by
$$\dR_\mu(\alpha, \beta) = \sH(\alpha \given \beta) + \sH(\beta \given \alpha).$$
We refer the reader to Appendix \ref{sec:metric} for some of the basic properties of this metric.

Let $(X, \mu)$ be a probability space, and let $\cF$ be a sub-$\sigma$-algebra. Let $\pi : (X, \mu) \rightarrow (Y, \nu)$ be the factor associated to $\cF$, and let $\mu = \int \mu_y \ d \nu(y)$ be the disintegration of $\mu$ over $\nu$. For a countable Borel partition $\alpha$ of $X$, the \emph{conditional Shannon entropy} of $\alpha$ relative to $\cF$ is
$$\sH(\alpha \given \cF) = \int_Y \sum_{A \in \alpha} -\mu_y(A) \cdot \log \mu_y(A) \ d \nu(y) = \int_Y \sH_{\mu_y}(\alpha) \ d \nu(y).$$
When necessary, we will write $\sH_\mu(\alpha \given \cF)$ to emphasize the measure. For a partition $\beta$ of $X$ we set $\sH(\alpha \given \beta) = \sH(\alpha \given \salg(\beta))$. For $B \subseteq X$ we write
$$\sH_B(\alpha \given \cF) = \sH_{\mu_B}(\alpha \given \cF),$$
where $\mu_B$ is the normalized restriction of $\mu$ to $B$ defined by $\mu_B(A) = \mu(A \cap B) / \mu(B)$. Note that if $\beta \subseteq \cF$ is a countable partition of $X$ then
$$\sH(\alpha \given \cF) = \sum_{B \in \beta} \mu(B) \cdot \sH_B(\alpha \given \cF).$$
In particular, $\sH(\alpha \given \beta) = \sum_{B \in \beta} \mu(B) \cdot \sH_B(\alpha)$.

We will need the following standard properties of Shannon entropy (proofs can be found in \cite{Do11}, specifically Equation 1.3.2 and Facts 1.6.24, 1.6.27, 1.6.38, and 1.6.39):

\begin{lem} \label{lem:shan}
Let $(X, \mu)$ be a standard probability space, let $\alpha$ and $\beta$ be countable Borel partitions of $X$, and let $\cF$, $\Sigma$, and $(\cF_n)_{n \in \N}$ be sub-$\sigma$-algebras. Then
\begin{enumerate}
\item[\rm (i)] $\sH(\alpha \given \cF) = 0$ if and only if $\alpha \subseteq \cF$ mod null sets;
\item[\rm (ii)] $\sH(\alpha \given \cF) \leq \log |\alpha|$;
\item[\rm (iii)] if $\alpha \geq \beta$ then $\sH(\alpha \given \cF) \geq \sH(\beta \given \cF)$;
\item[\rm (iv)] if $\Sigma \subseteq \cF$ then $\sH(\alpha \given \Sigma) \geq \sH(\alpha \given \cF)$;
\item[\rm (v)] $\sH(\alpha \vee \beta \given \cF) = \sH(\beta \given \cF) + \sH(\alpha \given \beta \vee \cF)$;
\item[\rm (vi)] if $\sH(\alpha) < \infty$ then $\sH(\alpha \given \cF) = \sH(\alpha)$ if and only if $\alpha$ and $\cF$ are independent;
\item[\rm (vii)] $\sH(\alpha \given \cF) = \sup_\beta \sH(\beta \given \cF)$ where $\beta$ ranges over all finite partitions coarser than $\alpha$;
\item[\rm (viii)] $\sH(\alpha \given \bigvee_{n \in \N} \cF_n) = \inf_{n \in \N} \sH(\alpha \given \cF_n)$ if the $\cF_n$'s are increasing and the right-hand side is finite;
\item[\rm (ix)] if $\sum_n \sH(\alpha_n) < \infty$ then $\bigvee_{n \in \N} \alpha_n$ is essentially countable and $\sH(\bigvee_{n \in \N} \alpha_n) \leq \sum_n \sH(\alpha_n)$.
\end{enumerate}
\end{lem}

A \emph{pre-partition} of $X$ is a collection of pairwise-disjoint subsets of $X$. We say that a partition $\beta$ extends a pre-partition $\alpha$, written $\beta \sqsupseteq \alpha$, if there is an injection $\iota : \alpha \rightarrow \beta$ with $A \subseteq \iota(A)$ for every $A \in \alpha$. Equivalently, $\beta \sqsupseteq \alpha$ if and only if the restriction of $\beta$ to $\cup \alpha$ coincides with $\alpha$. For a Borel pre-partition $\alpha$, we define the \emph{reduced $\sigma$-algebra} $\ralg_G(\alpha)$ to be the collection of Borel sets $R \subseteq X$ such that there is a conull $X' \subseteq X$ satisfying:
\begin{quote}
for every $r \in R \cap X'$ and $x \in X' \setminus R$ there is $g \in G$ with $g \cdot r, g \cdot x \in \cup \alpha$ and with $g \cdot r$ and $g \cdot x$ lying in distinct classes of $\alpha$.
\end{quote}
It is a basic exercise to verify that $\ralg_G(\alpha)$ is indeed a $\sigma$-algebra. We note two basic lemmas related to reduced $\sigma$-algebras which we will need.

\begin{lem}[{\cite[Lem.2.2]{S14}}] \label{LEM EXT}
Let $G \acts (X, \mu)$ be a {\pmp} action, and let $\alpha$ be a pre-partition. If $\beta$ is a partition and $\beta \sqsupseteq \alpha$ then $\salg_G(\beta) \supseteq \ralg_G(\alpha)$.
\end{lem}

\begin{lem} \label{LEM RFACT}
Let $G \acts (X, \mu)$ be a {\pmp} action and let $G \acts (Y, \nu)$ be a factor of $(X, \mu)$ under the map $\pi : (X, \mu) \rightarrow (Y, \nu)$. If $\alpha$ is a countable pre-partition of $Y$ then $\pi^{-1}(\ralg_G(\alpha)) \subseteq \ralg_G(\pi^{-1}(\alpha))$.
\end{lem}

\begin{proof}
Fix $S \in \ralg_G(\alpha)$ and set $R = \pi^{-1}(S)$. By definition, there is a conull $Y' \subseteq Y$ so that for all $s \in S \cap Y'$ and all $y \in Y' \setminus S$ there is $g \in G$ with $g \cdot s, g \cdot y \in \cup \alpha$ and with $g \cdot s$ and $g \cdot y$ lying in distinct classes of $\alpha$. Let $X'$ be the conull set $\pi^{-1}(Y')$ and pick any $r \in R \cap X'$ and $x \in X' \setminus R$. Then $\pi(r) \in S \cap Y'$ and $\pi(x) \in Y' \setminus S$. So there is $g \in G$ with $\pi(g \cdot r), \pi(g \cdot x) \in \cup \alpha$ and with $\pi(g \cdot r)$ and $\pi(g \cdot x)$ lying in distinct classes of $\alpha$. Clearly then $g \cdot r, g \cdot x \in \cup \pi^{-1}(\alpha)$ and $g \cdot r$ and $g \cdot x$ are in distinct classes of $\pi^{-1}(\alpha)$. Therefore $R \in \ralg_G(\pi^{-1}(\alpha))$.
\end{proof}

For a {\pmp} action $G \acts (X, \mu)$ and a $G$-invariant sub-$\sigma$-algebra $\cF$, we let $\cJ$ denote the $\sigma$-algebra of $G$-invariant sets and we define the \emph{relative Rokhlin entropy} of $G \acts (X, \mu)$ relative to $\cF$, denoted $\rh_G(X, \mu \given \cF)$, as
$$\inf \Big\{ \sH(\alpha \given \cF \vee \cJ) \: \alpha \text{ is a countable Borel partition and } \salg_G(\alpha) \vee \cF \vee \cJ = \Borel(X) \Big\}.$$
Since we only work with ergodic actions here, $\cJ$ will always be trivial and hence
$$\rh_G(X, \mu \given \cF) = \inf \Big\{ \sH(\alpha \given \cF) : \alpha \text{ a countable partition and } \salg_G(\alpha) \vee \cF = \Borel(X)\}.$$
When $G$ is amenable and the action is free, the relative Rokhlin entropy coincides with relative Kolmogorov--Sinai entropy \cite{S14,AS}. Additionally, similar to the Rudolph--Weiss theorem \cite{RW00}, it is known that $\rh_G(X, \mu \given \cF)$ is invariant under orbit equivalences for which the orbit-change cocycle is $\cF$-measurable \cite{S14}. The following is the strongest version of the main theorem from Part I \cite{S14}.

\begin{thm} \label{thm:relbasic}
Let $G \acts (X, \mu)$ be a {\pmp} ergodic action with $(X, \mu)$ non-atomic, and let $\cF$ be a $G$-invariant sub-$\sigma$-algebra. If $\xi$ is a countable Borel partition of $X$, $0 < r \leq 1$, and $\pv$ is a probability vector with $\sH(\xi \given \cF) < r \cdot \sH(\pv)$, then there is a Borel pre-partition $\alpha = \{A_i : 0 \leq i < |\pv|\}$ with $\mu(\cup \alpha) = r$, $\mu(A_i) = r p_i$ for every $0 \leq i < |\pv|$, and $\salg_G(\xi) \vee \cF \subseteq \ralg_G(\alpha) \vee \cF$.
\end{thm}

For a {\pmp} ergodic action $G \acts (X, \mu)$, a collection $\mathcal{C}$ of Borel sets, and a $G$-invariant sub-$\sigma$-algebra $\cF$, we define the \emph{outer Rokhlin entropy} as
$$\rh_{G, \mu}(\mathcal{C} \given \cF) = \inf \Big\{ \sH(\alpha \given \cF) : \alpha \text{ is a countable Borel partition and } \mathcal{C} \subseteq \salg_G(\alpha) \vee \cF \Big\}.$$
When $\cF = \{X, \varnothing\}$ we simply write $\rh_{G, \mu}(\mathcal{C})$ for $\rh_{G, \mu}(\mathcal{C} \given \cF)$. If $G \acts (Y, \nu)$ is a factor of $(X, \mu)$, then we define $\rh_{G, \mu}(Y, \nu) = \rh_{G, \mu}(\Sigma)$, where $\Sigma$ is the $G$-invariant sub-$\sigma$-algebra of $X$ associated to $Y$.

A fundamental property of Rokhlin entropy is that it is countably sub-additive. This fact will be critical to nearly all the main theorems of this paper.

\begin{cor} \label{cor:add2}
Let $G \acts (X, \mu)$ be a {\pmp} ergodic action, let $\mathcal{C} \subseteq \Borel(X)$, let $\Sigma$ be a $G$-invariant sub-$\sigma$-algebra, and let $(\cF_n)_{n \in \N}$ be an increasing sequence of $G$-invariant sub-$\sigma$-algebras with $\mathcal{C} \subseteq \bigvee_{n \in \N} \cF_n \vee \Sigma$. Then
\begin{equation} \label{eqn:add2}
\rh_{G, \mu}(\mathcal{C} \given \Sigma) \leq \rh_{G, \mu}(\cF_1 \given \Sigma) + \sum_{n \geq 2} \rh_{G, \mu}(\cF_n \given \cF_{n-1} \vee \Sigma).
\end{equation}
\end{cor}

Note that we do not assume that $\mu$ is non-atomic, and note that one may choose to have $\cF_n = \cF_{n-1}$ for all large $n$ (in which case the sum becomes finite).

\begin{proof}
Assume that $\rh_{G, \mu}(\mathcal{C} \given \Sigma) > 0$ and that the right-hand side of (\ref{eqn:add2}) is finite, as otherwise there is nothing to show.

If $\mu$ has an atom, then by ergodicity $X$ is finite. Note that a partition consisting of a single point and its complement is both generating and of minimum (non-zero) Shannon entropy. This furthermore remains true when working relative to a $G$-invariant sub-$\sigma$-algebra. Therefore for any $\mathcal{D} \subseteq \Borel(X)$ and $G$-invariant sub-$\sigma$-algebra $\Psi$, $\rh_{G, \mu}(\mathcal{D} \given \Psi)$ is $0$ if $\mathcal{D} \subseteq \Psi$ and otherwise is the minimum of $\sH(\alpha \given \Psi)$ among all partitions with $\sH(\alpha \given \Psi) > 0$. From this observation, we see that the first non-zero term on the right-hand side of (\ref{eqn:add2}) is equal to $\rh_{G, \mu}(\mathcal{C} \given \Sigma)$.

Now assume that $\mu$ is non-atomic. Denote the value of the right-hand side of (\ref{eqn:add2}) by $h$. Fix $\epsilon > 0$. Set $\cF_0 = \{X, \varnothing\}$. For each $n \in \N$ fix a partition $\beta_n'$ with $\sH(\beta_n' \given \cF_{n-1} \vee \Sigma) < \rh_{G, \mu}(\cF_n \given \cF_{n-1} \vee \Sigma) + \epsilon / 2^n$ and $\cF_n \subseteq \salg_G(\beta_n') \vee \cF_{n-1} \vee \Sigma$. Apply Theorem \ref{thm:relbasic} to obtain a partition $\beta_n$ with $\sH(\beta_n) < \rh_{G, \mu}(\cF_n \given \cF_{n-1} \vee \Sigma) + \epsilon / 2^n$ and with $\cF_n \subseteq \salg_G(\beta_n) \vee \cF_{n-1} \vee \Sigma$. Then $\sum_{n \in \N} \sH(\beta_n) < h + \epsilon < \infty$, so by Lemma \ref{lem:shan} $\beta = \bigvee_{n \in \N} \beta_n$ is essentially countable and $\sH(\beta) < h + \epsilon$. We have
$$\mathcal{C} \subseteq \bigvee_{n \in \N} \cF_n \vee \Sigma \subseteq \bigvee_{n \in \N} \salg_G(\beta_n) \vee \Sigma \subseteq \salg_G(\beta) \vee \Sigma$$
and thus $\rh_{G, \mu}(\mathcal{C} \given \Sigma) \leq \sH(\beta) < h + \epsilon$. Now let $\epsilon$ tend to $0$.
\end{proof}

In the remainder of the paper, we will simply refer to Corollary \ref{cor:add2} as the property of \emph{sub-additivity}.

We mention one last fact we will need.

\begin{thm}[Seward--Tucker-Drob \cite{ST14}] \label{thm:robin}
Let $G$ be a countably infinite group and let $G \acts (X, \mu)$ be a free {\pmp} action. Then for every $\epsilon > 0$ there is a factor $G \acts (Y, \nu)$ of $(X, \mu)$ such that $\rh_G(Y, \nu) < \epsilon$ and $G$ acts freely on $Y$.
\end{thm}

\section{Translations and independence} \label{sec:trans}

In this section we show that if the Rokhlin entropy of a free ergodic action is finite and is realized by a generating partition, then the action is isomorphic to a Bernoulli shift.

We recall the following well known lemma. This lemma is a special case of a more general result due to Kechris--Solecki--Todorcevic \cite[Prop. 4.2 and Prop. 4.5]{KST99}.

\begin{lem} \label{lem:marker}
Let $G \acts (X, \mu)$ be a {\pmp} action. If $Y \subseteq X$ is Borel and $T \subseteq G$ is finite, then there exists a Borel set $D \subseteq Y$ such that $Y \subseteq T^{-1} T \cdot D$ and $T \cdot d \cap T \cdot d' = \varnothing$ for all $d \neq d' \in D$.
\end{lem}

\begin{lem} \label{lem:part}
Let $G$ be a countably infinite group, let $G \acts (X, \mu)$ be a free {\pmp} action, and let $T \subseteq G$ be finite. Then there is a Borel partition $\xi$ of $X$ such that for every $C \in \xi$ we have $\mu(C) \geq \frac{1}{4} \cdot |T|^{-4}$ and $t \cdot C \cap s \cdot C = \varnothing$ for all $t \neq s \in T$.
\end{lem}

\begin{proof}
If $|T| = 1$ then by setting $\xi = \{X\}$ we are done. So assume $|T| \geq 2$. Since the action is free, the condition $t \cdot C \cap s \cdot C = \varnothing$ for all $t \neq s \in T$ is equivalent to the condition $T \cdot c \cap T \cdot c' = \varnothing$ for all $c \neq c' \in C$. By repeatedly applying Lemma \ref{lem:marker} we can inductively construct disjoint sets $C_1, C_2, \ldots$ such that for every $i$
$$X \setminus (C_1 \cup C_2 \cup \cdots \cup C_{i-1}) \subseteq T^{-1} T \cdot C_i$$
and $T \cdot c \cap T \cdot c' = \varnothing$ for all $c \neq c' \in C_i$. We claim that there is $n \leq |T^{-1} T| + 1$ such that $X = C_1 \cup \cdots \cup C_n$. If not, then there is $x \in X \setminus (C_1 \cup \cdots \cup C_{|T^{-1} T|+1})$. Then $x \in T^{-1} T \cdot C_i$ for every $i$ and hence $T^{-1} T \cdot x$ meets every $C_i$, $1 \leq i \leq |T^{-1} T| + 1$. This contradicts the $C_i$'s being disjoint.

Set $\xi = \{C_i \: 1 \leq i \leq n\}$. If $\mu(C_i) < \frac{1}{4} \cdot |T|^{-4}$ for some $i$, then since $\xi$ is a partition of $X$ with $|\xi| \leq 2 |T|^2$, there must be some $j$ with $\mu(C_j) > \frac{1}{2} |T|^{-2}$. So
$$\mu \Big( C_j \setminus T^{-1} T \cdot C_i \Big) \geq \frac{1}{2 |T|^2} - \frac{|T|^2}{4 |T|^4} = \frac{1}{4 |T|^2} > 2 \cdot \frac{1}{4 |T|^4}.$$
Thus by removing from $C_j$ a subset $B \subseteq C_j \setminus T^{-1} T \cdot C_i$ having measure $\mu(B) = \frac{1}{4} \cdot |T|^{-4}$ and by enlarging $C_i$ to contain $B$, we will have reduced the number of sets in $\xi$ having measure less than $\frac{1}{4} \cdot |T|^{-4}$. This process can be repeated until every set in $\xi$ has measure at least $\frac{1}{4} \cdot |T|^{-4}$.
\end{proof}

We are ready for the main result of this section.

\begin{thm} \label{thm:drop}
Let $G$ be a countably infinite group, let $G \acts (X, \mu)$ be a free {\pmp} ergodic action, and let $\cF$ be a $G$-invariant sub-$\sigma$-algebra. If $\alpha$ is a countable partition, $T \subseteq G$ is finite, $\epsilon > 0$, and $\frac{1}{|T|} \cdot \sH(\alpha^T \given \cF) < \sH(\alpha \given \cF) - \epsilon$, then $\rh_{G, \mu}(\alpha \given \cF) < \sH(\alpha \given \cF) - \epsilon / (16 |T|^3)$.
\end{thm}

\begin{proof}
By invariance of $\mu$ and $\cF$, $\sH(\alpha^{s T} \given \cF) = \sH(\alpha^T \given \cF)$ for all $s \in G$. So by replacing $T$ with a translate $s T$ we may assume that $1_G \in T$. By Theorem \ref{thm:robin}, there is a factor $G \acts (Z, \eta)$ of $(X, \mu)$ such that the action of $G$ on $Z$ is free and $\rh_G(Z, \eta) < \epsilon / (16 \cdot |T|^3)$. Let $\Sigma$ be the $G$-invariant sub-$\sigma$-algebra of $X$ associated to $Z$. If $\sH(\alpha \given \cF \vee \Sigma) \leq \sH(\alpha \given \cF) - \epsilon / 2$, then by sub-additivity (Corollary \ref{cor:add2})
\begin{align*}
\rh_{G, \mu}(\alpha \given \cF) & \leq \rh_{G, \mu}(\Sigma \given \cF) + \rh_{G, \mu}(\alpha \given \cF \vee \Sigma)\\
 & \leq \rh_G(Z, \eta) + \sH(\alpha \given \cF \vee \Sigma)\\
 & < \frac{\epsilon}{16 \cdot |T|^3} + \sH(\alpha \given \cF) - \frac{\epsilon}{2}\\
 & < \sH(\alpha \given \cF) - \frac{\epsilon}{16 |T|^3},
\end{align*}
and thus we are done. So assume $\sH(\alpha \given \Sigma \vee \cF) > \sH(\alpha \given \cF) - \epsilon / 2$. Note that
$$\frac{1}{|T|} \cdot \sH(\alpha^T \given \cF \vee \Sigma) \leq \frac{1}{|T|} \cdot \sH(\alpha^T \given \cF) < \sH(\alpha \given \cF) - \epsilon < \sH(\alpha \given \cF \vee \Sigma) - \epsilon / 2.$$
By definition the action $G \acts (Z, \eta)$ is free. So we can apply Lemma \ref{lem:part} to obtain a partition $\xi \subseteq \Sigma$ of $X$ such that for every $C \in \xi$ we have $t^{-1} \cdot C \cap s^{-1} \cdot C = \varnothing$ for all $t \neq s \in T$ and $\mu(C) \geq \frac{1}{4} \cdot |T|^{-4}$.

Let $\pi : (X, \mu) \rightarrow (Y, \nu)$ be the factor associated to $\cF \vee \Sigma$, and let $\mu = \int \mu_y \ d \nu(y)$ be the disintegration of $\mu$ over $\nu$. We have
\begin{align*}
\sum_{C \in \xi} & \int_{\pi(C)} \left( \sum_{t \in T} \sH_{\mu_y}(t \cdot \alpha) - \sH_{\mu_y}(\alpha^T) \right) \ d \nu(y) \\
 & = \int_Y \left( \sum_{t \in T} \sH_{\mu_y}(t \cdot \alpha) - \sH_{\mu_y}(\alpha^T) \right) \ d \nu(y) \\
 & = \sum_{t \in T} \sH(t \cdot \alpha \given \cF \vee \Sigma) - \sH(\alpha^T \given \cF \vee \Sigma) \\
 & = |T| \cdot \sH(\alpha \given \cF \vee \Sigma) - \sH(\alpha^T \given \cF \vee \Sigma) \\
 & > |T| \cdot \frac{\epsilon}{2}.
\end{align*}
So we can fix $D \in \xi$ with
$$\int_{\pi(D)} \left( \sum_{t \in T} \sH_{\mu_y}(t \cdot \alpha) - \sH_{\mu_y}(\alpha^T) \right) \ d \nu(y) > |T| \cdot \frac{\epsilon}{2} \cdot \mu(D).$$
Set $R = T^{-1} \cdot D$ and observe that $\mu(R) = |T| \cdot \mu(D)$. Note that for almost-every $y \in Y$ and all $g \in G$ we have $\mu_y(E) = \mu_{g \cdot y}(g \cdot E)$ for Borel $E \subseteq X$ and hence also $\sH_{\mu_y}(\alpha) = \sH_{\mu_{g \cdot y}}(g \cdot \alpha)$. Thus
\begin{align*}
\sH & _R (\alpha \given \cF \vee \Sigma) - \frac{1}{|T|} \cdot \sH_D(\alpha^T \given \cF \vee \Sigma) \\
 & = \frac{1}{\mu(R)} \cdot \int_{T^{-1} \cdot \pi(D)} \sH_{\mu_y}(\alpha) \ d \nu(y) - \frac{1}{|T| \cdot \mu(D)} \cdot \int_{\pi(D)} \sH_{\mu_y}(\alpha^T) \ d \nu(y) \\
 & = \frac{1}{|T| \cdot \mu(D)} \cdot \sum_{t \in T} \int_{t^{-1} \cdot \pi(D)} \sH_{\mu_y}(\alpha) \ d \nu(y) - \frac{1}{|T| \cdot \mu(D)} \cdot \int_{\pi(D)} \sH_{\mu_y}(\alpha^T) \ d \nu(y) \\
 & = \frac{1}{|T| \cdot \mu(D)} \cdot \int_{\pi(D)} \left( \sum_{t \in T} \sH_{\mu_y}(t \cdot \alpha) - \sH_{\mu_y}(\alpha^T) \right) \ d \nu(y) \\
 & > \frac{\epsilon}{2}.
\end{align*}

Define a new partition
$$\beta = \Big( \alpha \res (X \setminus R) \Big) \cup \Big\{ R \setminus D \Big\} \cup \Big( \alpha^T \res D \Big).$$
Observe that $D \subseteq R$ since $1_G \in T$. Let $\gamma$ be the partition of $X$ consisting of the sets $t^{-1} \cdot D$, $t \in T$, and $X \setminus R$. Then $\gamma \subseteq \Sigma$ and $\alpha$ is coarser than
$$\alpha \vee \gamma = \Big( \alpha \res (X \setminus R) \Big) \cup \bigcup_{t \in T} \Big( \alpha \res t^{-1} \cdot D \Big).$$
Since $\alpha \res (X \setminus R) \subseteq \beta$ and for each $t \in T$ the partition $t \cdot (\alpha \res t^{-1} \cdot D) = (t \cdot \alpha \res D)$ of $D$ is coarser than $\alpha^T \res D$, we see that
$$\alpha \leq \alpha \vee \gamma \subseteq \salg_G(\beta) \vee \Sigma.$$
Therefore $\rh_{G, \mu}(\alpha \given \cF \vee \Sigma) \leq \sH(\beta \given \cF \vee \Sigma)$.

Since $R, D \in \Sigma$ and $\mu(R) = |T| \cdot \mu(D) \geq \frac{1}{4} \cdot |T|^{-3}$ we have
\begin{align*}
\sH(\beta \given \cF \vee \Sigma) & = \mu(X \setminus R) \cdot \sH_{X \setminus R}(\alpha \given \cF \vee \Sigma) + \mu(D) \cdot \sH_D(\alpha^T \given \cF \vee \Sigma)\\
 & = \mu(X \setminus R) \cdot \sH_{X \setminus R}(\alpha \given \cF \vee \Sigma) + \mu(R) \cdot \frac{1}{|T|} \cdot \sH_D(\alpha^T \given \cF \vee \Sigma)\\
 & < \mu(X \setminus R) \cdot \sH_{X \setminus R}(\alpha \given \cF \vee \Sigma) + \mu(R) \cdot \sH_R(\alpha \given \cF \vee \Sigma) - \mu(R) \cdot \frac{\epsilon}{2}\\
 & = \sH(\alpha \given \cF \vee \Sigma) - \mu(R) \cdot \frac{\epsilon}{2}\\
 & \leq \sH(\alpha \given \cF \vee \Sigma) - \frac{\epsilon}{8 |T|^3}
\end{align*}
Therefore
\begin{align*}
\rh_{G, \mu}(\alpha \given \cF \vee \Sigma) + \rh_G(Z, \eta) & \leq \sH(\beta \given \cF \vee \Sigma) + \rh_G(Z, \eta)\\
 & < \sH(\alpha \given \cF \vee \Sigma) - \frac{\epsilon}{8 |T|^3} + \frac{\epsilon}{16 \cdot |T|^3}\\
 & \leq \sH(\alpha \given \cF) - \frac{\epsilon}{16 |T|^3}.
\end{align*}
Thus we are done by sub-additivity of Rokhlin entropy.
\end{proof}

We will also need the following variant of Theorem \ref{thm:drop} where we replace both instances of $\sH(\alpha \given \cF)$ with $\sH(\alpha)$.

\begin{cor} \label{cor:drop2}
Let $G$ be a countably infinite group, let $G \acts (X, \mu)$ be a free {\pmp} ergodic action, and let $\cF$ be a $G$-invariant sub-$\sigma$-algebra. If $\alpha$ is a countable partition, $T \subseteq G$ is finite, $\epsilon > 0$, and $\frac{1}{|T|} \cdot \sH(\alpha^T \given \cF) < \sH(\alpha) - \epsilon$, then $\rh_{G, \mu}(\alpha \given \cF) < \sH(\alpha) - \epsilon / (32 |T|^3)$.
\end{cor}

\begin{proof}
If $\sH(\alpha \given \cF) < \sH(\alpha) - \epsilon / 2$ then clearly
$$\rh_{G, \mu}(\alpha \given \cF) \leq \sH(\alpha \given \cF) < \sH(\alpha) - \frac{\epsilon}{32 |T|^3}.$$
So suppose that $\sH(\alpha \given \cF) \geq \sH(\alpha) - \epsilon / 2$. Then
$$\sH(\alpha^T \given \cF) < |T| \cdot \sH(\alpha) - |T| \cdot \epsilon \leq |T| \cdot \sH(\alpha \given \cF) - |T| \cdot \epsilon / 2.$$
In this case we can apply Theorem \ref{thm:drop}.
\end{proof}

We recall the simple fact that a free ergodic {\pmp} action $G \acts (X, \mu)$ is isomorphic to a Bernoulli shift if and only if there is a generating partition whose $G$-translates are mutually independent.

\begin{cor} \label{cor:bernoulli}
Let $G$ be a countably infinite group and let $G \acts (X, \mu)$ be a free {\pmp} ergodic action. If $\alpha$ is a generating partition with $\sH(\alpha) = \rh_G(X, \mu) < \infty$ then $G \acts (X, \mu)$ is isomorphic to a Bernoulli shift and $\alpha$ is a Bernoulli generating partition.
\end{cor}

\begin{proof}
Since $\rh_G(X, \mu) = \sH(\alpha)$, Theorem \ref{thm:drop} implies that $\sH(\alpha^T) = |T| \cdot \sH(\alpha)$ for every finite $T \subseteq G$. Since $\sH(\alpha) < \infty$, this implies that the $G$-translates of $\alpha$ are mutually independent. As $\alpha$ is a generating partition, it follows that $G \acts (X, \mu)$ is isomorphic to a Bernoulli shift.
\end{proof}

With the above corollary we answer a question of N.-P. Chung in \cite[Question 5.4]{Ch13} regarding equilibrium states for sofic pressure. We refer the reader to \cite{Ch13} for the relevant definitions.

\begin{cor} \label{cor:pressure}
Let $G$ be a sofic group, let $L$ be a finite set, and let $f_0 : L \rightarrow \R$ be a function. Consider the Bernoulli shift $L^G$ and define $f : L^G \rightarrow \R$ by $f(x) = f_0(x(1_G))$. Define a probability measure $\lambda$ on $L$ by
$$\lambda(\ell) = \frac{\exp(f_0(\ell))}{\sum_{\ell' \in L} \exp(f_0(\ell'))}.$$
Then $\lambda^G$ is the unique equilibrium state for $f$ for every sofic approximation $\Sigma$ to $G$.
\end{cor}

\begin{proof}
Chung proved that $\lambda^G$ is an equilibrium state, and he proved that it is the unique equilibrium state among Bernoulli measures. So we only need to show that every equilibrium state is a Bernoulli measure. Denote by $P_\Sigma(f, L^G, G)$ the sofic pressure of $f$ with respect to a sofic approximation $\Sigma$ to $G$. Let $\mu$ be a $G$-invariant probability measure on $L^G$ which is an equilibrium state for $f$ and $\Sigma$, meaning
\begin{equation} \label{eqn:eqlb}
P_\Sigma(f, L^G, G) = h^\Sigma_G(L^G, \mu) + \int f \ d \mu,
\end{equation}
where $h^\Sigma_G(L^G, \mu)$ denotes the sofic entropy with respect to $\Sigma$. Define a measure $\nu$ on $L$ by $\nu(\ell) = \mu(\{ x \in L^G : x(1_G) = \ell\})$. Let $\cL = \{R_\ell : \ell \in L\}$ be the canonical generating partition of $L^G$, where $R_\ell = \{x \in L^G : x(1_G) = \ell\}$. Then $\int f \ d \nu^G = \int f \ d \mu$ and
\begin{equation} \label{eqn:pres}
h_G^\Sigma(L^G, \nu^G) = \sH(L, \nu) = \sH_\nu(\cL) = \sH_\mu(\cL) \geq \rh_G(L^G, \mu) \geq h_G^\Sigma(L^G, \mu).
\end{equation}
However, by the variational principle \cite{Ch13} we have
$$P_\Sigma(f, L^G, G) \geq h_G^\Sigma(L^G, \nu^G) + \int f \ d \nu^G.$$
Combining this with (\ref{eqn:eqlb}) and (\ref{eqn:pres}), we conclude that $\rh_G(L^G, \mu) = \sH_\mu(\cL)$. By Corollary \ref{cor:bernoulli}, $\cL$ is a Bernoulli generating partition for $(L^G, \mu)$ and thus $\mu = \nu^G$.
\end{proof}

\section{Gottschalk's surjunctivity conjecture and CPE} \label{sec:gott}

In this section we relate the Rokhlin entropy values of Bernoulli shifts with Gottschalk's surjunctivity conjecture and the property of completely positive entropy.

\begin{cor} \label{cor:gott}
Let $G$ be a countably infinite group. Assume that $\rh_G(k^G, u_k^G) = \log(k)$ for every $k \in \N$. Then $G$ satisfies Gottschalk's surjunctivity conjecture and Kaplansky's direct finiteness conjecture.
\end{cor}

\begin{proof}
We verify Gottschalk's surjunctivity conjecture as Kaplansky's direct finiteness conjecture will then hold automatically \cite[Section I.5]{CL13}. Let $k \geq 2$ and let $\phi : k^G \rightarrow k^G$ be a continuous $G$-equivariant injection. Set $(Y, \nu) = (\phi(k^G), \phi_*(u_k^G))$ where $\nu = \phi_*(u_k^G)$ is the push-forward measure. Let $\cL = \{R_i : 0 \leq i < k\}$ denote the canonical generating partition for $k^G$, where
$$R_i = \{x \in k^G : x(1_G) = i\}.$$
Note that $\cL \res Y$ is generating for $Y$. Since $\phi$ is injective, it is an isomorphism between $(k^G, u_k^G)$ and $(Y, \nu)$. Therefore
$$\log(k) = \rh_G(k^G, u_k^G) = \rh_G(Y, \nu) \leq \sH_\nu(\cL) \leq \log | \cL | = \log(k).$$
So $\rh_G(Y, \nu) = \sH_\nu(\cL) = \log(k)$. In particular, $\sH_\nu(\cL^T) = |T| \cdot \sH_\nu(\cL)$ for all finite $T \subseteq G$ by Theorem \ref{thm:drop}.

Towards a contradiction, suppose that $\phi$ is not surjective. Then its image is a proper closed subset of $k^G$ and hence there is some finite $T \subseteq G$ and $w \in k^{T^{-1}}$ such that $y \res T^{-1} \neq w$ for all $y \in Y$. This implies that $|\cL^T \res Y| \leq k^{|T|} - 1$. So
$$\sH_\nu(\cL^T) \leq \log | \cL^T \res Y | \leq \log(k^{|T|} - 1) < |T| \cdot \log(k) = |T| \cdot \sH_\nu(\cL),$$
a contradiction.
\end{proof}

Next we consider the property of completely positive outer Rokhlin entropy. We say that an ergodic action $G \acts (X, \mu)$ has \emph{completely positive outer Rokhlin entropy} if every factor $G \acts (Y, \nu)$ which is non-trivial (i.e. $Y$ is not a single point) satisfies $\rh_{G, \mu}(Y, \nu) > 0$.

\begin{cor} \label{cor:cpe}
Let $G$ be a countably infinite group. Assume that $\rh_G(L^G, \lambda^G) = \sH(L, \lambda)$ for every probability space $(L, \lambda)$. Then every Bernoulli shift over $G$ has completely positive outer Rokhlin entropy.
\end{cor}

\begin{proof}
Let $(L, \lambda)$ be a probability space, and let $G \acts (Y, \nu)$ be a non-trivial factor of $(L^G, \lambda^G)$. Let $\cF$ be the $G$-invariant sub-$\sigma$-algebra of $L^G$ associated to $(Y, \nu)$.

First we mention a short proof in the case that $\sH(L, \lambda) < \infty$. Let $\cL$ be the canonical partition of $L^G$. If $T \subseteq G$ is finite and $\sH(\cL^T \given \cF) = \sH(\cL^T) = |T| \cdot \sH(\cL)$, then $\cL^T$ must be independent of $\cF$ by Lemma \ref{lem:shan}. Since $\cL$ is a generating partition, this cannot occur for every finite $T \subseteq G$. So by Theorem \ref{thm:drop} we get $\rh_G(L^G, \lambda^G \given \cF) < \sH(\cL) = \rh_G(L^G, \lambda^G)$. Therefore by sub-additivity
$$\rh_{G, \lambda^G}(Y, \nu) \geq \rh_G(L^G, \lambda^G) - \rh_G(L^G, \lambda^G \given \cF) > 0.$$
Here we only needed to assume $\rh_G(L^G, \lambda^G) = \sH(L, \lambda) < \infty$ for this fixed choice of $(L, \lambda)$. In the general case below, we must assume that $\rh_G(L^G, \lambda^G) = \sH(L, \lambda)$ for all probability spaces $(L, \lambda)$.

Fix an increasing sequence of finite partitions $\cL_k$ of $L$ with $\bigvee_{k \in \N} \salg(\cL_k) = \Borel(L)$, and let $(L_k, \lambda_k)$ denote the factor of $(L, \lambda)$ associated to $\cL_k$. Let $\cL = \{R_\ell : \ell \in L\}$ be the canonical partition of $L^G$, where $R_\ell = \{x \in L^G : x(1_G) = \ell\}$. We identify each of the partitions $\cL_k$ as coarsenings of $\cL \subseteq \Borel(L^G)$. Note that $(L_k^G, \lambda_k^G)$ is the factor of $(L^G, \lambda^G)$ associated to $\salg_G(\cL_k)$. When working with $L_k^G$, for $m \leq k$ we view $\cL_m$ as a partition of $L_k^G$ in the natural way. Note that by our assumption and by sub-additivity
\begin{align*}
\sH(L_k, \lambda_k) = \rh_G(L_k^G, \lambda_k^G) & \leq \rh_{G, \lambda_k^G}(\cL_m) + \rh_G(L_k^G, \lambda_k^G \given \salg_G(\cL_m))\\
 & \leq \sH(\cL_m) + \sH(\cL_k \given \cL_m) = \sH(\cL_k) = \sH(L_k, \lambda_k).
\end{align*}
So equality holds throughout and
\begin{equation} \label{eqn:cpe1}
\rh_G(L_k^G, \lambda_k^G) = \sH(\cL_m) + \rh_G(L_k^G, \lambda_k^G \given \salg_G(\cL_m)).
\end{equation}

Fix a non-trivial finite partition $\cP \subseteq \cF$ and fix $0 < \epsilon < \sH(\cP) / 9$. By Corollary \ref{cor:invgen} there is $m \in \N$, finite $T \subseteq G$, and $\beta \leq \cL_m^T$ with $\dR_{\lambda^G}(\beta, \cP) < \epsilon$. Note that $\rh_{G, \lambda^G}(\cP) \leq \sH(\cP) < \infty$. Fix a partition $\cQ$ with
$$\sH(\cQ) < \rh_{G, \lambda^G}(\cP) + \frac{\epsilon}{32 |T|^3} \leq \rh_{G, \lambda^G}(Y, \nu) + \frac{\epsilon}{32 |T|^3}$$
and with $\cP \subseteq \salg_G(\cQ)$. By Corollary \ref{cor:gen} there is a finite $W \subseteq G$ and $\cP' \leq \cQ^W$ with $\dR_{\lambda^G}(\cP', \cP) < \epsilon$. Since $\sH(\cQ) < \infty$, we can apply Corollary \ref{cor:invgen} and Lemma \ref{lem:coarse} to get $k \geq m$, $\gamma \leq \salg_G(\cL_k)$ with $\dR_{\lambda^G}(\gamma, \cQ) < \epsilon / (32 |T|^3)$, and $\beta' \leq \gamma^W$ with $\dR_{\lambda^G}(\beta', \cP') < \epsilon$. Note that
\begin{equation} \label{eqn:cpe2}
\rh_{G, \lambda_k^G}(\gamma) \leq \sH(\gamma) < \sH(\cQ) + \dR_{\lambda^G}(\gamma, \cQ) \leq \rh_{G, \lambda^G}(Y, \nu) + \frac{2 \epsilon}{32 |T|^3}.
\end{equation}
Also note that $\sH(\cL_m^T \given \beta) = \sH(\cL_m^T) - \sH(\beta)$ since $\beta \leq \cL_m^T$.

We have
\begin{align*}
\frac{1}{|T|} \cdot \sH(\cL_m^T \given \beta') & < \frac{1}{|T|} \cdot \sH(\cL_m^T \given \beta) + \frac{1}{|T|} \cdot 2 \dR_{\lambda^G}(\beta', \beta)\\
 & < \sH(\cL_m) - \frac{1}{|T|} \cdot \sH(\beta) + \frac{6 \epsilon}{|T|}.
\end{align*}
So by Corollary \ref{cor:drop2}
\begin{equation} \label{eqn:cpe3}
\rh_{G, \lambda_k^G}(\cL_m \given \salg_G(\gamma)) < \sH(\cL_m) - \frac{1}{32 |T|^3} \cdot \sH(\beta) + \frac{6 \epsilon}{32 |T|^3}.
\end{equation}
By sub-additivity we have
$$\rh_G(L_k^G, \lambda_k^G) \leq \rh_{G, \lambda_k^G}(\gamma) + \rh_{G, \lambda_k^G}(\cL_m \given \salg_G(\gamma)) + \rh_G(L_k^G, \lambda_k^G \given \salg_G(\cL_m)).$$
Combining this inequality with (\ref{eqn:cpe1}) and then (\ref{eqn:cpe3}) gives
$$\rh_{G, \lambda_k^G}(\gamma) \geq \sH(\cL_m) - \rh_{G, \lambda_k^G}(\cL_m \given \salg_G(\gamma)) \geq \frac{1}{32 |T|^3} \cdot \sH(\beta) - \frac{6 \epsilon}{32 |T|^3}.$$
Finally, using (\ref{eqn:cpe2}) we conclude
\begin{equation*}
\rh_{G, \lambda^G}(Y, \nu) > \frac{1}{32 |T|^3} \cdot \sH(\beta) - \frac{8 \epsilon}{32 |T|^3} > \frac{1}{32 |T|^3} \cdot \sH(\cP) - \frac{9 \epsilon}{32 |T|^3} > 0.\qedhere
\end{equation*}
\end{proof}

\section{Approximately Bernoulli partitions} \label{sec:approx}

For a {\pmp} action $G \acts (X, \mu)$ we let $E_G^X$ denote the induced orbit equivalence relation:
$$E_G^X = \{(x, y) \: \exists g \in G, \ \ g \cdot x = y\}.$$
The \emph{pseudo-group} of $E_G^X$, denoted $[[E_G^X]]$, is the set of all Borel bijections $\theta : \dom(\theta) \rightarrow \rng(\theta)$ where $\dom(\theta), \rng(\theta) \subseteq X$ are Borel and $\theta(x) \in G \cdot x$ for every $x \in \dom(\theta)$. Note that since $G$ acts measure preservingly and $\theta(x) \in G \cdot x$ for all $x \in \dom(\theta)$, $\theta$ is measure-preserving as well.

\begin{defn}
Let $G \acts (X, \mu)$ be a {\pmp} action, let $\theta \in [[E_G^X]]$, and let $\cF$ be a $G$-invariant sub-$\sigma$-algebra. We say that $\theta$ is \emph{$\cF$-expressible} if $\dom(\theta), \rng(\theta) \in \cF$ and there is a $\cF$-measurable partition $\{Z_g^\theta \: g \in G\}$ of $\dom(\theta)$ such that $\theta(x) = g \cdot x$ for every $x \in Z_g^\theta$ and all $g \in G$.
\end{defn}

We will need the following two simple lemmas from Part I \cite{S14}.

\begin{lem}[{\cite[Lem 3.2]{S14}}] \label{lem:expmove}
Let $G \acts (X, \mu)$ be a {\pmp} action and let $\cF$ be a $G$-invariant sub-$\sigma$-algebra. If $\theta \in [[E_G^X]]$ is $\cF$-expressible and $A \subseteq X$, then $\theta(A) = \theta(A \cap \dom(\theta))$ is $\salg_G(\{A\}) \vee \cF$-measurable. In particular, if $A \in \cF$ then $\theta(A) \in \cF$.
\end{lem}

\begin{lem}[{\cite[Lem 3.3]{S14}}] \label{lem:expgroup}
Let $G \acts (X, \mu)$ be a {\pmp} action and let $\cF$ be a $G$-invariant sub-$\sigma$-algebra. If $\theta, \phi \in [[E_G^X]]$ are $\cF$-expressible then so are $\theta^{-1}$ and $\theta \circ \phi$.
\end{lem}

In this section we will show how to construct generating partitions which are approximately Bernoulli. The result of this section will be key in order to study the Rokhlin entropy values of Bernoulli shifts. We begin with a few lemmas.

\begin{lem} \label{lem:pack}
Let $G \acts (X, \mu)$ be a {\pmp} ergodic action, let $\cF$ be a $G$-invariant sub-$\sigma$-algebra, and let $B \in \cF$ with $\mu(B) > 0$. Then there is a finite collection $\Phi \subseteq [[E_G^X]]$ of $\cF$-expressible functions such that $\{\dom(\phi) : \phi \in \Phi\}$ partitions $X$ and $\rng(\phi) \subseteq B$ for every $\phi \in \Phi$.
\end{lem}

\begin{proof}
We claim that there is a finite partition $\gamma \subseteq \cF$ with $\mu(C) \leq \mu(B)$ for every $C \in \gamma$. If the factor $G \acts (Y, \nu)$ of $(X, \mu)$ associated to $\cF$ is purely atomic then we can simply let $\gamma$ be the pre-image of the partition of $Y$ into points. On the other hand, if $(Y, \nu)$ is non-atomic then we can find such a partition in $Y$ and let $\gamma$ be its pre-image. Now by \cite[Lemma 3.5]{S14}, for every $C \in \gamma$ there is an $\cF$-expressible $\phi_C \in [[E_G^X]]$ with $\dom(\phi_C) = C$ and $\rng(\phi_C) \subseteq B$. Then $\Phi = \{\phi_C : C \in \gamma\}$ has the desired properties.
\end{proof}

\begin{lem} \label{lem:subset}
Let $G \acts (X, \mu)$ be a {\pmp} ergodic action with $(X, \mu)$ non-atomic, let $\cF$ be a $G$-invariant sub-$\sigma$-algebra, and let $B \in \cF$. If $\xi$ is a countable partition of $X$ and $\pv = (p_i)$ is a probability vector with
$$\sH(\xi \given \cF) < \mu(B) \cdot \sH(\pv),$$
then there is a partition $\alpha = \{A_i : 0 \leq i < |\pv|\}$ of $B$ with $\mu(A_i) = p_i \cdot \mu(B)$ for every $0 \leq i < |\pv|$ and with $\xi \subseteq \salg_G(\alpha') \vee \cF$ for every partition $\alpha'$ of $X$ extending $\alpha$.
\end{lem}

\begin{proof}
Let $\Phi \subseteq [[E_G^X]]$ be as given by Lemma \ref{lem:pack}. For $\phi \in \Phi$, define a partition $\xi_\phi$ of $X$ by
$$\xi_\phi = \Big\{ X \setminus \rng(\phi) \Big\} \cup \phi \Big( \xi \res \dom(\phi) \Big),$$
and set $\zeta = \bigvee_{\phi \in \Phi} \xi_\phi$. Note that $\zeta$ is countable since $\Phi$ is finite. Also observe that
\begin{equation} \label{eqn:subset}
\mu(\rng(\phi)) \cdot \sH_{\rng(\phi)}(\xi_\phi \given \cF) = \mu(\dom(\phi)) \cdot \sH_{\dom(\phi)}(\xi \given \cF)
\end{equation}
since $\phi$ is measure-preserving and $\phi(\cF \res \dom(\phi)) = \cF \res \rng(\phi)$ by Lemmas \ref{lem:expmove} and \ref{lem:expgroup}.

We claim that $\xi \subseteq \salg_G(\zeta) \vee \cF$. Consider $C \in \xi$ and $\phi \in \Phi$. Since $\phi$ is $\cF$-expressible, we have $\rng(\phi) \in \cF$. Thus $\xi_\phi \res \rng(\phi) \subseteq \salg_G(\zeta) \vee \cF$. It follows from Lemmas \ref{lem:expmove} and \ref{lem:expgroup} that
$$\phi^{-1}(\xi_\phi \res \rng(\phi)) \subseteq \salg_G(\zeta) \vee \cF.$$
Since $C \cap \dom(\phi)$ is an element of the set on the left, and since $C$ is the union of $C \cap \dom(\phi)$ for $\phi \in \Phi$, we conclude that $\xi \subseteq \salg_G(\zeta) \vee \cF$.

For $g \in G$ define $\gamma_g \in [[E_G^X]]$ with $\dom(\gamma_g) = \rng(\gamma_g) = B$ by the rule
$$\gamma_g(x) = y \Longleftrightarrow y = g^i \cdot x \text{ where } i > 0 \text{ is least with } g^i \cdot x \in B.$$
By the Poincar\'{e} recurrence theorem, the domain and range of $\gamma_g$ are indeed conull in $B$. Note that $\gamma_g$ is $\cF$-expressible since $B \in \cF$. Let $\Gamma$ be the group of transformations of $B$ generated by $\{\gamma_g : g \in G\}$. Then every $\gamma \in \Gamma$ is $\cF$ expressible by Lemma \ref{lem:expgroup}. Let $\mu_B$ denote the normalized restriction of $\mu$ to $B$, so that $\mu_B(A) = \mu(A \cap B) / \mu(B)$. Since $\mu$ is ergodic, it is not difficult to check that the action of $\Gamma$ on $(B, \mu_B)$ is ergodic. Similarly, since $\mu$ is non-atomic $\mu_B$ is non-atomic as well. Using (\ref{eqn:subset}) and the fact that $\dom(\phi), \rng(\phi) \in \cF$, we have
\begin{align*}
\mu(B) \cdot \sH_{\mu_B}(\zeta \given \cF) & = \mu(B) \cdot \sH_B(\zeta \given \cF)\\
 & \leq \sum_{\phi \in \Phi} \mu(B) \cdot \sH_B(\xi_\phi \given \cF)\\
 & = \sum_{\phi \in \Phi} \mu(B) \cdot \sH_B(\xi_\phi \given \{\rng(\phi), X \setminus \rng(\phi)\} \vee \cF)\\
 & = \sum_{\phi \in \Phi} \mu(\rng(\phi)) \cdot \sH_{\rng(\phi)}(\xi_\phi \given \cF)\\
 & = \sum_{\phi \in \Phi} \mu(\dom(\phi)) \cdot \sH_{\dom(\phi)}(\xi \given \cF)\\
 & = \sH(\xi \given \cF)\\
 & < \mu(B) \cdot \sH(\pv).
\end{align*}
So by Theorem \ref{thm:relbasic} there is a partition $\alpha = \{A_i : 0 \leq i < |\pv|\}$ of $B$ with $\mu_B(A_i) = p_i$ for every $0 \leq i < |\pv|$ and with $\zeta \res B \subseteq \salg_\Gamma(\alpha) \vee \cF$. Since $\zeta \res (X \setminus B)$ is trivial and $X \setminus B \in \cF$, it follows that $\zeta \subseteq \salg_\Gamma(\alpha) \vee \cF$.

Since $A_i \subseteq B$ and $\mu_B(A_i) = p_i$, it follows that $\mu(A_i) = p_i \cdot \mu(B)$. Now let $\alpha'$ be a partition of $X$ extending $\alpha$. Since $\Gamma$ is $\cF$-expressible, it follows from Lemma \ref{lem:expmove} that $\salg_G(\alpha') \vee \cF$ is $\Gamma$-invariant. Since also $B \in \cF$ and $\alpha = \alpha' \res B$, we have $\salg_\Gamma(\alpha) \vee \cF \subseteq \salg_G(\alpha') \vee \cF$. Therefore $\zeta \subseteq \salg_G(\alpha') \vee \cF$ and hence
\begin{equation*}
\xi \subseteq \salg_G(\zeta) \vee \cF \subseteq \salg_G(\alpha') \vee \cF. \qedhere
\end{equation*}
\end{proof}

The following lemma is, in some ways, a strengthening of Theorem \ref{thm:relbasic}.

\begin{lem} \label{lem:superrel}
Let $G \acts (X, \mu)$ be a {\pmp} ergodic action with $(X, \mu)$ non-atomic, let $\cF$ be a $G$-invariant sub-$\sigma$-algebra, and let $\xi$ be a countable Borel partition of $X$. If $\beta \subseteq \cF$ is a collection of pairwise disjoint Borel sets and $\{\pv^B : B \in \beta\}$ is a collection of probability vectors with
$$\sH(\xi \given \cF) < \sum_{B \in \beta} \mu(B) \cdot \sH(\pv^B),$$
then there is a partition $\alpha = \{A_i : 0 \leq i < |\alpha|\}$ of $\cup \beta$ with $\mu(A_i \cap B) = p^B_i \cdot \mu(B)$ for every $B \in \beta$ and $0 \leq i < |\alpha|$ and with $\xi \subseteq \salg_G(\alpha') \vee \cF$ for every partition $\alpha'$ of $X$ extending $\alpha$.
\end{lem}

\begin{proof}
Without loss of generality, we may assume that $\beta$ consists of non-null sets and that each probability vector $\pv^B$ is non-trivial. Fix $\epsilon > 0$ with
$$\sH(\xi \given \cF) < \sum_{B \in \beta} \mu(B) \cdot \sH(\pv^B) - \epsilon \cdot \mu(\cup \beta).$$
For each $B \in \beta$, fix any probability vector $\qv^B$ satisfying
$$\mu(B) \cdot \sH(\pv^B) - \epsilon \cdot \mu(B) < \sH(\qv^B) < \mu(B) \cdot \sH(\pv^B).$$
Let $\rv$ be the probability vector which represents the independent join of the $\qv^B$'s. Specifically, $\rv = (r_\pi)_{\pi \in \N^\beta}$ where
$$r_\pi = \prod_{B \in \beta} q^B_{\pi(B)}.$$
Then
$$\sH(\rv) = \sum_{B \in \beta} \sH(\qv^B) > \sum_{B \in \beta} \mu(B) \cdot \sH(\pv^B) - \epsilon \cdot \mu(\cup \beta) > \sH(\xi \given \cF).$$
So by Theorem \ref{thm:relbasic} there is a partition $\gamma = \{C_\pi : \pi \in \N^\beta\}$ with $\xi \subseteq \salg_G(\gamma) \vee \cF$ and with $\mu(C_\pi) = r_\pi$ for every $\pi \in \N^\beta$.

For each $B \in \beta$, let $\gamma^B$ be the coarsening of $\gamma$ associated to $\qv^B$. Specifically, $\gamma^B = \{C^B_i : 0 \leq i < |\qv^B|\}$ where
$$C^B_i = \bigcup_{\substack{\pi \in \N^\beta\\\pi(B) = i}} C_\pi.$$
Note that $\gamma = \bigvee_{B \in \beta} \gamma^B$. Also note that $\mu(C^B_i) = q^B_i$ and $\sH(\gamma^B) = \sH(\qv^B) < \mu(B) \cdot \sH(\pv^B)$. For each $B \in \beta$ we apply Lemma \ref{lem:subset} to $\gamma^B$ in order to obtain a partition $\alpha^B = \{A^B_i : 0 \leq i < |\pv^B|\}$ of $B$ with $\mu(A^B_i) = \mu(B) \cdot p^B_i$ and $\gamma^B \subseteq \salg_G(\zeta) \vee \cF$ for every partition $\zeta$ of $X$ extending $\alpha^B$. Now define $\alpha = \{A_i : 0 \leq i < |\alpha|\}$ where $A_i = \bigcup_{B \in \beta} A^B_i$. Then for $B \in \beta$ and $0 \leq i < |\alpha|$ we have $\mu(A_i \cap B) = \mu(A_i^B) = p_i^B \cdot \mu(B)$. Furthermore, if $\alpha'$ is a partition of $X$ which extends $\alpha$, then $\alpha'$ extends every $\alpha^B$ and hence $\gamma^B \subseteq \salg_G(\alpha') \vee \cF$. It follows that
\begin{equation*}
\xi \subseteq \salg_G(\gamma) \vee \cF \subseteq \salg_G(\alpha') \vee \cF. \qedhere
\end{equation*}
\end{proof}

We will need the result of Ab\'{e}rt and Weiss that all free actions weakly contain Bernoulli shifts \cite{AW13}. The following is a slightly modified statement of their result, obtained by invoking \cite[Lemma 5]{AW13} and performing a perturbation.

\begin{thm}[Ab\'{e}rt--Weiss \cite{AW13}]
Let $G \acts (X, \mu)$ be a {\pmp} free action, and let $\pv = (p_i)$ be a finite probability vector. If $T \subseteq G$ is finite and $\epsilon > 0$, then there is a partition $\gamma = \{C_i : 0 \leq i < |\pv|\}$ of $X$ such that $\mu(C_i) = p_i$ for every $0 \leq i < |\pv|$ and $\sH(\gamma^T) / |T| > \sH(\gamma) - \epsilon$.
\end{thm}

We are almost ready to construct approximately Bernoulli generating partitions. For this construction we will find it more convenient to use Borel partitions of $([0, 1], \lambda)$, where $\lambda$ is Lebesgue measure, in place of probability vectors. We first make a simple observation.

\begin{lem} \label{lem:upmap}
If $\cQ \leq \cP$ are finite partitions of $([0, 1], \lambda)$ and $0 < r < \sH(\cP \given \cQ)$, then there is a finite partition $\cR$ such that $\cQ \leq \cR$ and $\sH(\cP \given \cR) = r$.
\end{lem}

\begin{proof}
Fix a $\dR_\lambda$-continuous $1$-parameter family of finite partitions $\cQ_t$, $0 \leq t \leq 1$, such that $\cQ_0 = \cQ$, $\cQ_1 = \cP$, and $\cQ \leq \cQ_t$ for all $t$. The function $t \mapsto \sH(\cP \given \cQ_t)$ is continuous, $\sH(\cP \given \cQ_0) = \sH(\cP \given \cQ) > r$, and $\sH(\cP \given \cQ_1) = \sH(\cP \given \cP) = 0$. Therefore there is $t \in (0, 1)$ with $\sH(\cP \given \cQ_t) = r$. Set $\cR = \cQ_t$.
\end{proof}

For countable partitions $\alpha$ and $\beta$ of $(X, \mu)$ we define
$$\dB_\mu(\alpha, \beta) = \inf \Big\{ \mu(Y) \: Y \subseteq X \text{ and } \alpha \res (X \setminus Y) = \beta \res (X \setminus Y) \Big\}.$$
The function $\dB_\mu$ defines a metric on the space of countable partitions, and in fact for every $n \in \N$ the restrictions of $\dB_\mu$ and $\dR_\mu$ to the space of $n$-piece partitions are uniformly equivalent \cite[Fact 1.7.7]{Do11}. We will temporarily need to use this metric in the proof of the next theorem.

Recall that for a countable ordered partition $\alpha = \{A_i : 0 \leq i < |\alpha|\}$ we let $\dist(\alpha)$ denote the probability vector having $i^{\text{th}}$ term $\mu(A_i)$. For $B \subseteq X$ we also write $\dist_B(\alpha)$ for the probability vector having $i^{\text{th}}$ term $\mu(A_i \cap B) / \mu(B)$.

\begin{thm} \label{thm:fake}
Let $G$ be a countably infinite group and let $G \acts (X, \mu)$ be a free {\pmp} ergodic action. Let $\cP$ and $\cQ$ be ordered countable partitions of $([0, 1], \lambda)$ with $\cQ \leq \cP$ and $\sH(\cP) < \infty$. If $\rh_G(X, \mu) < \sH(\cP \given \cQ)$, then for every finite $T \subseteq G$ and $\epsilon > 0$ there is an ordered generating partition $\alpha$ with $\dist(\alpha) = \dist(\cP)$,
$$\frac{1}{|T|} \cdot \sH(\alpha^T) > \sH(\alpha) - \epsilon,$$
and $\rh_{G, \mu}(\beta) < \epsilon$, where $\beta$ is the coarsening of $\alpha$ corresponding to $\cQ \leq \cP$.
\end{thm}

\begin{proof}
First assume that $\cP$ is finite. Apply Lemma \ref{lem:upmap} to obtain a finite partition $\cR$ of $[0, 1]$ which is finer than $\cQ$ and satisfies
$$\rh_G(X, \mu) < \sH(\cP \given \cR) < \rh_G(X, \mu) + \frac{\epsilon}{256 \cdot |T|^3}.$$
Without loss of generality, we may assume that $\lambda(R) > 0$ for every $R \in \cR$. Set $s = \min_{R \in \cR} \lambda(R)$. Since $\dB_\mu$ and $\dR_\mu$ are uniformly equivalent on the space of partitions of $X$ having at most $|\cP|$ pieces, there is
$$0 < \kappa < \frac{\epsilon}{256 \cdot |T|^3 \cdot \sH(\cP)}$$
satisfying
$$\rh_G(X, \mu) < (1 - \kappa) \cdot \sH(\cP \given \cR)$$
such that $\dR_\mu(\xi, \xi') < \epsilon / 8$ whenever $\xi$ and $\xi'$ are partitions of $X$ with at most $|\cP|$ pieces and with $\dB_\mu(\xi, \xi') \leq \kappa$.

By Theorem \ref{thm:robin}, there is a factor $G \acts (Y, \nu)$ of $(X, \mu)$ such that
$$\rh_G(Y, \nu) < s \kappa \cdot \sH(\cP) < \frac{\epsilon}{256 \cdot |T|^3}$$
and $G$ acts freely on $(Y, \nu)$. Let $\cF$ be the sub-$\sigma$-algebra of $X$ associated to $(Y, \nu)$. Note that by sub-additivity
$$\rh_G(X, \mu) \leq \rh_G(X, \mu \given \cF) + \rh_G(Y, \nu) < \rh_G(X, \mu \given \cF) + \frac{\epsilon}{256 \cdot |T|^3}.$$
Therefore
\begin{equation} \label{eqn:fake1}
\sH(\cP \given \cR) < \rh_G(X, \mu) + \frac{\epsilon}{256 \cdot |T|^3} < \rh_G(X, \mu \given \cF) + \frac{\epsilon}{128 \cdot |T|^3}.
\end{equation}
Since $G$ acts freely on $(Y, \nu)$, the Ab\'{e}rt--Weiss theorem implies that there is an ordered partition $\gamma = \{C_k : 0 \leq k < |\cR|\} \subseteq \cF$ with $\dist(\gamma) = \dist(\cR)$ and
\begin{equation} \label{eqn:fake6}
\frac{1}{|T|} \cdot \sH(\gamma^T) > \sH(\gamma) - \frac{\epsilon}{2}.
\end{equation}

By construction $\rh_G(Y, \nu) < s \kappa \cdot \sH(\cP)$. So by applying Theorem \ref{thm:relbasic} to $(Y, \nu)$ (and invoking Lemma \ref{LEM RFACT}) we obtain a set $Z_0 \in \cF$ with $\mu(Z_0) = s \kappa$ and a partition $\alpha^0 = \{A_i^0 : 0 \leq i < |\cP|\} \subseteq \cF$ of $Z_0$ with $\cF \subseteq \ralg_G(\alpha^0)$ and
\begin{equation} \label{eqn:fake2}
\mu(A_i^0) = s \kappa \cdot \lambda(P_i) = \mu(Z_0) \cdot \lambda(P_i)
\end{equation}
for every $0 \leq i < |\cP|$. Note that
$$\mu(Z_0 \cap C_k) \leq \mu(Z_0) = s \kappa \leq \kappa \cdot \lambda(R_k) = \kappa \cdot \mu(C_k)$$
for all $0 \leq k < |\cR|$ since $\dist(\gamma) = \dist(\cR)$. Since $(Y, \nu)$ is non-atomic and $\{Z_0\} \cup \gamma \subseteq \cF$, it follows from the above inequality that there exists $Z_1 \in \cF$ such that $Z_1 \cap Z_0 = \varnothing$, $\mu(Z_1) = 1 - \kappa$, and $\mu(Z_1 \cap C) = (1 - \kappa) \cdot \mu(C)$ for every $C \in \gamma$.

Consider the collection $\gamma \res Z_1$ of pairwise disjoint sets. For each $C_k \cap Z_1 \in \gamma \res Z_1$ define the probability vector $\pv^{C_k \cap Z_1} = \dist_{R_k}(\cP)$. We have
\begin{align*}
\rh_G(X, \mu \given \cF) & \leq \rh_G(X, \mu)\\
 & < (1 - \kappa) \cdot \sH(\cP \given \cR)\\
 & = \sum_{0 \leq k < |\cR|} (1 - \kappa) \lambda(R_k) \cdot \sH_{R_k}(\cP)\\
 & = \sum_{0 \leq k < |\cR|} \mu(C_k \cap Z_1) \cdot \sH(\pv^{C_k \cap Z_1}).
\end{align*}
So by Lemma \ref{lem:superrel}, there is a partition $\alpha^1 = \{A_i^1 : 0 \leq i < |\cP|\}$ of $Z_1$ with
\begin{equation} \label{eqn:fake9}
\mu(A_i^1 \cap C_k \cap Z_1) = \frac{\lambda(R_k \cap P_i)}{\lambda(R_k)} \cdot \mu(C_k \cap Z_1) = (1 - \kappa) \cdot \lambda(R_k \cap P_i)
\end{equation}
for every $i$ and $k$ and with $\salg_G(\alpha') \vee \cF = \Borel(X)$ for all partitions $\alpha'$ extending $\alpha^1$. Note that
\begin{equation} \label{eqn:fake3}
\mu(A_i^1) = (1 - \kappa) \cdot \lambda(P_i) = \mu(Z_1) \cdot \lambda(P_i)
\end{equation}
for every $i$.

Set $Z_2 = X \setminus (Z_0 \cup Z_1)$. Pick any partition $\alpha^2 = \{A_i^2 : 0 \leq i < |\cP|\}$ of $Z_2$ with
\begin{equation} \label{eqn:fake4}
\mu(A_i^2) = \lambda(P_i) \cdot \mu(Z_2)
\end{equation}
for every $i$. Set $\alpha = \{A_i : 0 \leq i < |\cP|\}$ where $A_i = A_i^0 \cup A_i^1 \cup A_i^2$. Then $\mu(A_i) = \lambda(P_i)$ for every $i$ by (\ref{eqn:fake2}), (\ref{eqn:fake3}), and (\ref{eqn:fake4}). Additionally, $\alpha$ extends $\alpha^0$ and thus $\cF \subseteq \salg_G(\alpha)$ by Lemma \ref{LEM EXT}. Similarly, $\alpha$ extends $\alpha^1$ so
$$\Borel(X) = \salg_G(\alpha) \vee \cF = \salg_G(\alpha).$$
Thus $\alpha$ is generating.

In order to check that $\rh_{G,\mu}(\beta) < \epsilon$, where $\beta$ is the coarsening of $\alpha$ corresponding to $\cQ \leq \cP$, we will temporarily work with a perturbation $\alpha^*$ of $\alpha$. By (\ref{eqn:fake9}), the partition $\alpha \vee \gamma$ almost has the same distribution as $\cP \vee \cR$. We perturb $\alpha$ so that the joint distribution with $\gamma$ will be precisely the distribution of $\cP \vee \cR$. Using (\ref{eqn:fake9}), we may pick a partition $\alpha^* = \{A_i^* : 0 \leq i < |\cP|\}$ extending $\alpha^1$ and satisfying $\mu(A_i^* \cap C_k) = \lambda(P_i \cap R_k)$ for all $0 \leq i < |\cP|$ and $0 \leq k < |\cR|$. Then $\dist(\alpha) = \dist(\alpha^*) = \dist(\cP)$ and $\dB_\mu(\alpha, \alpha^*) \leq \mu(Z_0 \cup Z_2) = \kappa$. It follows from the definition of $\kappa$ that $\dR_\mu(\alpha, \alpha^*) < \epsilon / 8$ and thus by (\ref{eqn:fake1})
\begin{align}
\sH(\alpha \given \gamma) & < \sH(\alpha^* \given \gamma) + \epsilon / 8\nonumber\\
 & = \sH(\cP \given \cR) + \epsilon / 8\nonumber\\
 & < \rh_G(X, \mu \given \cF) + \epsilon / 4\nonumber\\
 & \leq \sH(\alpha \given \cF) + \epsilon / 4.\label{eqn:fake5}
\end{align}
Let $\beta$ and $\beta^*$ be the coarsenings of $\alpha$ and $\alpha^*$, respectively, corresponding to the coarsening $\cQ$ of $\cP$. Since $\mu(A_i^* \cap C_k) = \lambda(P_i \cap R_k)$ for all $i$ and $k$, there is an isomorphism $(X, \mu) \rightarrow ([0, 1], \lambda)$ of measure spaces which identifies $\alpha^*$ with $\cP$ and $\gamma$ with $\cR$. Since $\cQ$ is coarser than $\cR$, it follows that $\beta^*$ is coarser than $\gamma$. So $\beta^* \subseteq \cF$ and hence $\rh_{G, \mu}(\beta^*) \leq \rh_G(Y, \nu) < \epsilon / 8$. Additionally, $\dB_\mu(\alpha, \alpha^*) \leq \kappa$ implies $\dB_\mu(\beta, \beta^*) \leq \kappa$ and thus $\dR_\mu(\beta, \beta^*) < \epsilon / 8$. It follows that $\sH(\beta \given \beta^*) < \epsilon / 8$ and hence $\rh_{G, \mu}(\beta) < \epsilon / 4 < \epsilon$ as required.

Finally, we check that $\sH(\alpha^T) / |T| > \sH(\alpha) - \epsilon$. Using (\ref{eqn:fake1}) and the fact that $Z_0, Z_1, Z_2 \in \cF$, we have
\begin{align*}
\sH(\alpha \given \cF) & = \mu(Z_0 \cup Z_2) \cdot \sH_{Z_0 \cup Z_2}(\alpha \given \cF) + \mu(Z_1) \cdot \sH_{Z_1}(\alpha \given \cF)\\
 & \leq \mu(Z_0 \cup Z_2) \cdot \sH_{Z_0 \cup Z_2}(\alpha) + \sH_{Z_1}(\alpha \given \gamma)\\
 & = \kappa \cdot \sH(\cP) + \sH(\cP \given \cR)\\
 & < \frac{\epsilon}{256 \cdot |T|^3} + \rh_G(X, \mu \given \cF) + \frac{\epsilon}{128 \cdot |T|^3}\\
 & < \rh_G(X, \mu \given \cF) + \frac{\epsilon}{64 \cdot |T|^3}
\end{align*}
Applying Theorem \ref{thm:drop}, we conclude that
$$\frac{1}{|T|} \cdot \sH(\alpha^T \given \gamma^T) \geq \frac{1}{|T|} \cdot \sH(\alpha^T \given \cF) \geq \sH(\alpha \given \cF) - \frac{\epsilon}{4}.$$
From the above inequality and (\ref{eqn:fake5}) we obtain
\begin{equation} \label{eqn:fake7}
\frac{1}{|T|} \cdot \sH(\alpha^T \given \gamma^T) > \sH(\alpha \given \gamma) - \frac{\epsilon}{2}.
\end{equation}
Also, we observe that
\begin{equation} \label{eqn:fake8}
\sH(\gamma^T \given \alpha^T) \leq \sum_{t \in T} \sH(t \cdot \gamma \given \alpha^T) \leq \sum_{t \in T} \sH(t \cdot \gamma \given t \cdot \alpha) = |T| \cdot \sH(\gamma \given \alpha).
\end{equation}
Therefore, using (\ref{eqn:fake6}), (\ref{eqn:fake7}), and (\ref{eqn:fake8}), we have
\begin{align*}
\frac{1}{|T|} \cdot \sH(\alpha^T) & = \frac{1}{|T|} \cdot \sH(\alpha^T \vee \gamma^T) - \frac{1}{|T|} \cdot \sH(\gamma^T \given \alpha^T)\\
 & = \frac{1}{|T|} \cdot \sH(\gamma^T) + \frac{1}{|T|} \cdot \sH(\alpha^T \given \gamma^T) - \frac{1}{|T|} \cdot \sH(\gamma^T \given \alpha^T)\\
 & > \sH(\gamma) - \epsilon / 2 + \sH(\alpha \given \gamma) - \epsilon / 2 - \sH(\gamma \given \alpha)\\
 & = \sH(\alpha \vee \gamma) - \epsilon - \sH(\gamma \given \alpha)\\
 & = \sH(\alpha) - \epsilon.
\end{align*}

To complete the proof, we consider the case where $\cP$ is countably infinite. By Lemma \ref{lem:shan}, there is a finite $\cQ_0 \leq \cQ$ so that $\sH(\cQ \given \cQ_0) < \epsilon / 2$. Note that $\rh_G(X, \mu) < \sH(\cP \given \cQ) \leq \sH(\cP \given \cQ_0)$. Now choose a finite $\cP_0 \leq \cP$ such that $\cQ_0 \leq \cP_0$, $\sH(\cP \given \cP_0) < \epsilon / 2$, and $\rh_G(X, \mu) < \sH(\cP_0 \given \cQ_0)$. Apply the above argument to get a generating partition $\alpha_0$ with $\dist(\alpha_0) = \dist(\cP_0)$, $\sH(\alpha_0^T) / |T| > \sH(\alpha_0) - \epsilon / 2$, and $\rh_{G, \mu}(\beta_0) < \epsilon / 2$, where $\beta_0$ is the coarsening of $\alpha_0$ corresponding to $\cQ_0$. Since $(X, \mu)$ is non-atomic, we may choose $\alpha \geq \alpha_0$ with $\dist(\alpha) = \cP$. Clearly $\alpha$ is still generating. Since $\sH(\alpha \given \alpha_0) = \sH(\cP \given \cP_0) < \epsilon / 2$, we have
$$\frac{1}{|T|} \cdot \sH(\alpha^T) \geq \frac{1}{|T|} \cdot \sH(\alpha_0^T) > \sH(\alpha_0) - \epsilon / 2 > \sH(\alpha) - \epsilon.$$
Finally, if $\beta$ is the coarsening of $\alpha$ corresponding to $\cQ$ then $\sH(\beta \given \beta_0) = \sH(\cQ \given \cQ_0) < \epsilon / 2$ and hence $\rh_{G, \mu}(\beta) < \rh_{G, \mu}(\beta_0) + \epsilon / 2 < \epsilon$.
\end{proof}

\section{Rokhlin entropy of Bernoulli shifts: Finite case} \label{sec:fin}

In this section we study the Rokhlin entropy of $(L^G, \lambda^G)$ when $\sH(L, \lambda) < \infty$. We first restate Theorem \ref{thm:fake} in terms of isomorphisms.

\begin{cor} \label{cor:iso}
Let $G$ be a countably infinite group and let $G \acts (X, \mu)$ be a free {\pmp} ergodic action. Let $(L, \lambda)$ be a probability space with $L$ finite. Let $\cL$ be the canonical partition of $L^G$, and let $\cK$ be a partition coarser than $\cL$. If $\rh_G(X, \mu) < \sH(\cL \given \cK)$, then for every open neighborhood $U \subseteq \E_G(L^G)$ of $\lambda^G$ and every $\epsilon > 0$, there is a $G$-equivariant isomorphism $\phi : (X, \mu) \rightarrow (L^G, \nu)$ with $\nu \in U$ and $\rh_{G, \nu}(\cK) < \epsilon$.
\end{cor}

\begin{proof}
By definition, $\cL = \{R_\ell : \ell \in L\}$ where
$$R_\ell = \{y \in L^G : y(1_G) = \ell\}.$$
Since $U$ is open, there are continuous functions $f_1, \ldots, f_n$ on $L^G$ and $\kappa_1 > 0$ such that for all $\nu \in \E_G(L^G)$
$$\Big| \textstyle{\int f_i \ d \lambda^G} - \textstyle{\int f_i \ d \nu} \Big| < \kappa_1 \text{ for all } 1 \leq i \leq n \Longrightarrow \nu \in U.$$
Since $L^G$ is compact, each $f_i$ is uniformly continuous and therefore there is a finite $T \subseteq G$ and continuous $\cL^T$-measurable functions $f_i'$ such that $\|f_i - f_i'\| < \kappa_1 / 2$ for each $1 \leq i \leq n$, where $\| \cdot \|$ denotes the sup-norm. Therefore there is $\kappa_2 > 0$ such that for all $\nu \in \E_G(L^G)$
$$\Big| \lambda^G(D) - \nu(D) \Big| < \kappa_2 \text{ for all } D \in \cL^T \Longrightarrow \nu \in U.$$
By viewing the restriction $\nu \res \cL^T$ as a $|\cL^T|$-tuple of real numbers from $[0, 1]$, we see that the quantity $|T| \cdot \sH_\nu(\cL) - \sH_\nu(\cL^T)$ is a continuous non-negative function of $\nu \res \cL^T$, and it is equal to $0$ if and only if the partitions $t \cdot \cL$, $t \in T$, are mutually $\nu$-independent. By compactness of $[0,1]^{|\cL^T|}$ and by $G$-invariance of $\nu$, it follows that there is $\kappa_3 > 0$ such that
$$\nu(R_\ell) = \lambda^G(R_\ell) \text{ for all } \ell \in L \text{ and } |T| \cdot \sH_\nu(\cL) - \sH_\nu(\cL^T) < \kappa_3 \Longrightarrow \nu \in U.$$

Now apply Theorem \ref{thm:fake} to obtain a generating partition $\alpha = \{A_\ell : \ell \in L\}$ of $X$ satisfying $\mu(A_\ell) = \lambda^G(R_\ell)$ for every $\ell \in L$, $\sH(\alpha^T) > |T| \cdot \sH(\alpha) - \kappa_3$, and $\rh_{G, \mu}(\beta) < \epsilon$, where $\beta$ is the coarsening of $\alpha$ corresponding to $\cK$. Since $\alpha$ is generating and its classes are indexed by $L$, it induces a $G$-equivariant isomorphism $\phi : (X, \mu) \rightarrow (L^G, \nu)$ which identifies $\alpha$ with $\cL$ and $\beta$ with $\cK$. We immediately have $\nu(R_\ell) = \mu(A_\ell) = \lambda^G(R_\ell)$ for every $\ell \in L$ and
$$|T| \cdot \sH_\nu(\cL) - \sH_\nu(\cL^T) = |T| \cdot \sH_\mu(\alpha) - \sH_\mu(\alpha^T) < \kappa_3.$$
So $\nu \in U$. Additionally, $\rh_{G, \nu}(\cK) = \rh_{G, \mu}(\beta) < \epsilon$.
\end{proof}

The key idea to understanding the Rokhlin entropy of $(L^G, \lambda^G)$ is to combine the approximations provided by the previous corollary with continuity properties of Rokhlin entropy. Here we develop only those continuity properties which are essential to studying $\rh_G(L^G, \lambda^G)$. A comprehensive study of the various continuity properties of Rokhlin entropy will be presented in Part III \cite{AS}. The results in Part III will in particular cover the case of actions which are not necessarily ergodic.

Recall that a real-valued function $f$ on a topological space $X$ is called \emph{upper-semicontinuous} if for every $x \in X$ and $\epsilon > 0$ there is an open set $U$ containing $x$ with $f(y) < f(x) + \epsilon$ for all $y \in U$. When $X$ is first countable, this is equivalent to saying that $f(x) \geq \limsup f(x_n)$ whenever $(x_n)$ is a sequence converging to $x$.

\begin{lem} \label{lem:mups}
Let $G$ be a countable group, let $L$ be a finite set, and let $L^G$ have the product topology. Let $\mathcal{C}$ be a countable collection of clopen sets, and let $\cF$ be the smallest $G$-invariant $\sigma$-algebra containing $\mathcal{C}$. Then the map $\mu \in \E_G(L^G) \mapsto \rh_G(L^G, \mu \given \cF)$ is upper-semicontinuous in the weak$^*$-topology.
\end{lem}

\begin{proof}
Let $\cL = \{R_\ell : \ell \in L\}$ be the canonical generating partition for $L^G$, where $R_\ell = \{x \in L^G : x(1_G) = \ell\}$. Fix a $G$-invariant probability measure $\mu$ on $L^G$ and fix $\epsilon > 0$. Pick a partition $\alpha$ satisfying $\sH_\mu(\alpha \given \cF) < \rh_G(L^G, \mu \given \cF) + \epsilon / 4$ and $\salg_G(\alpha) \vee \cF = \Borel(L^G)$ (equality up to $\mu$-null sets). Let $\gamma$ be a finite partition which is measurable with respect to the $G$-invariant algebra generated by $\mathcal{C}$ and let $T \subseteq G$ be a finite set satisfying
$$\sH_\mu(\cL \given \alpha^T \vee \gamma) < \epsilon / 4 \quad \text{and} \quad \sH_\mu(\alpha \given \gamma) < \rh_G(L^G, \mu \given \cF) + \epsilon / 4.$$
Since $\cL$ is a generating partition, there is a finite $W \subseteq G$ and a finite coarsening $\beta \leq \cL^W$ with $\dR_\mu(\beta, \alpha) < \epsilon / (8 |T|)$. Then
$$\sH_\mu(\cL \given \beta^T \vee \gamma) < \sH_\mu(\cL \given \alpha^T \vee \gamma) + 2 |T| \cdot \dR_\mu(\alpha, \beta) < \epsilon / 2$$
and
$$\sH_\mu(\beta \given \gamma) < \sH_\mu(\alpha \given \gamma) + \dR_\mu(\alpha, \beta) < \rh_G(L^G, \mu \given \cF) + \epsilon / 2.$$
Let $U$ be the set of $G$-invariant probability measures $\nu$ satisfying $\sH_\nu(\cL \given \beta^T \vee \gamma) < \epsilon / 2$ and $\sH_\nu(\beta \given \gamma) < \rh_G(L^G, \mu \given \cF) + \epsilon / 2$. Since $\cL$, $\beta$, and $\gamma$ are finite clopen partitions, the set $U$ is open and contains $\mu$. If $\nu \in U$ then by sub-additivity
\begin{equation*}
\rh_G(L^G, \nu \given \cF) \leq \sH_\nu(\beta \given \gamma) + \sH_\nu(\cL \given \beta^T \vee \gamma) < \rh_G(L^G, \mu \given \cF) + \epsilon.\qedhere
\end{equation*}
\end{proof}

We need one more continuity property of Rokhlin entropy.

\begin{lem} \label{lem:up}
Let $G \acts (X, \mu)$ be a {\pmp} ergodic action, let $\cF$ be a $G$-invariant sub-$\sigma$-algebra, and let $\alpha$ be a countable partition. Fix an increasing sequence of partitions $\alpha_n \leq \alpha$ with $\alpha = \bigvee_{n \in \N} \alpha_n$. For each $n$ let $G \acts (Y_n, \nu_n)$ be the factor of $(X, \mu)$ associated to $\salg_G(\alpha_n) \vee \cF$. Also let $\cF_n$ be the image of $\cF$ in $Y_n$. If $\sH(\alpha) < \infty$ and $\salg_G(\alpha) \vee \cF = \Borel(X)$ then
$$\rh_G(X, \mu \given \cF) = \lim_{n \rightarrow \infty} \rh_G(Y_n, \nu_n \given \cF_n).$$
\end{lem}

\begin{proof}
By sub-additivity, for every $n \in \N$ we have
$$\rh_G(X, \mu \given \cF) \leq \rh_G(Y_n, \nu_n \given \cF_n) + \sH(\alpha \given \alpha_n).$$
Since $\sH(\alpha \given \alpha_n)$ converges to $0$, we conclude $\rh_G(X, \mu \given \cF) \leq \liminf_{n \rightarrow \infty} \rh_G(Y_n, \nu_n \given \cF_n)$. Now fix $\epsilon > 0$ and let $\cP$ be a partition of $X$ satisfying $\sH(\cP \given \cF) < \rh_G(X, \mu \given \cF) + \epsilon / 6$ and $\salg_G(\cP) \vee \cF = \Borel(X)$. Pick a finite set $T \subseteq G$ with $\sH(\alpha \given \cP^T \vee \cF) < \epsilon / 6$. Let $\gamma' \subseteq \cF$ be a finite partition with $\sH(\cP \given \gamma') < \rh_G(X, \mu \given \cF) + \epsilon / 6$ and $\sH(\alpha \given \cP^T \vee \gamma') < \epsilon / 6$. Since $\salg_G(\alpha) \vee \cF = \Borel(X)$, we can find a finite partition $\gamma'' \subseteq \cF$, a finite $W \subseteq G$, and a coarsening $\cQ \leq \alpha^W \vee \gamma''$ such that $\dR_\mu(\cQ, \cP) < \epsilon / (24 |T|)$. Set $\gamma = \gamma' \vee \gamma''$. Let $n \in \N$ be sufficiently large so that $\sH(\alpha \given \alpha_n) < \epsilon / 6$ and so that there is a partition $\cQ_n \leq \alpha_n^W \vee \gamma$ with $\dR_\mu(\cQ_n, \cQ) < \epsilon / (24 |T|)$. Then $\dR_\mu(\alpha_n, \alpha) < \epsilon / 6$ and $\dR_\mu(\cQ_n, \cP) < \epsilon / (12 |T|)$. Therefore
$$\sH(\alpha_n \given \cQ_n^T \vee \gamma) < \sH(\alpha \given \cP^T \vee \gamma) + \dR_\mu(\alpha, \alpha_n) + 2 |T| \cdot \dR_\mu(\cQ_n, \cP) < \epsilon / 2$$
and
$$\sH(\cQ_n \given \gamma) < \sH(\cP \given \gamma) + \dR_\mu(\cQ_n, \cP) < \rh_G(X, \mu \given \cF) + \epsilon / 2.$$
So by sub-additivity
$$\rh_G(Y_n, \nu_n \given \cF_n) \leq \sH(\cQ_n \given \gamma) + \sH(\alpha_n \given \cQ_n^T \vee \gamma) < \rh_G(X, \mu \given \cF) + \epsilon.$$
This holds for all sufficiently large $n$ and all $\epsilon > 0$, completing the proof.
\end{proof}

Fix a countably infinite group $G$. Recall from the introduction the quantity
$$\suprh{G} = \sup_{G \acts (X, \mu)} \rh_G(X, \mu),$$
where the supremum is taken over all free ergodic {\pmp} actions $G \acts (X, \mu)$ with $\rh_G(X, \mu) < \infty$. If there is a free ergodic {\pmp} action $G \acts (X, \mu)$ with $\rh_G(X, \mu) = \infty$, we do not know if it necessarily follows that $\suprh{G} = \infty$. In particular, we do not know if $G \acts (X, \mu)$ must factor onto free actions having large but finite Rokhlin entropy values. However, we have the following.

\begin{lem} \label{lem:dfactor}
Let $G$ be a countably infinite group and let $G \acts (X, \mu)$ be a free {\pmp} ergodic action. If $\rh_G(X, \mu) < \infty$ then for every $0 \leq t \leq \rh_G(X, \mu)$ and $\delta > 0$ there is a factor $G \acts (Y, \nu)$ of $(X, \mu)$ such that $G$ acts freely on $Y$ and $\rh_G(Y, \nu) \in (t - \delta, t + \delta)$.
\end{lem}

\begin{proof}
Let $\pv$ be a probability vector with $\sH(\pv) = t$, and let $\qv$ be a probability vector with $\rh_G(X, \mu) - t < \sH(\qv) < \rh_G(X, \mu) - t + \delta$. Let $\rv$ be the probability vector which represents the independent join of $\pv$ and $\qv$. Specifically, $\rv = (r_{i,j})$ where $r_{i,j} = p_i \cdot q_j$. We have $\sH(\rv) = \sH(\pv) + \sH(\qv)$ so $\rh_G(X, \mu) < \sH(\rv)$. By Theorem \ref{intro:krieger} there is a generating partition $\gamma = \{C_{i,j}\}$ with $\mu(C_{i,j}) = r_{i,j}$. Let $\alpha = \{A_i \: 0 \leq i < |\pv|\}$ be the coarsening of $\gamma$ associated to $\pv$, meaning
$$A_i = \cup \{C_{i, j} \: 0 \leq j < |\qv|\}.$$
Similarly define $\beta = \{B_j \: 0 \leq j < |\qv|\}$ by
$$B_j = \cup \{C_{i, j} \: 0 \leq i < |\pv|\}.$$
Then $\dist(\alpha) = \pv$, $\dist(\beta) = \qv$, and $\alpha \vee \beta = \gamma$.

By Theorem \ref{thm:robin}, there is a free factor $G \acts (Z, \eta)$ of $(X, \mu)$ with $\rh_G(Z, \eta) < \delta$. Let $\zeta'$ be a generating partition for $Z$ with $\sH(\zeta') < \delta$, and let $\zeta$ be the pre-image of $\zeta'$ in $X$. Let $G \acts (Y, \nu)$ be the factor of $(X, \mu)$ associated to $\salg_G(\alpha \vee \zeta)$. Clearly $\alpha \vee \zeta$ pushes forward to a generating partition $\alpha' \vee \zeta''$ of $Y$ with $\sH(\alpha') = \sH(\pv)$ and $\sH(\zeta'') < \delta$. So $\rh_G(Y, \nu) \leq \sH(\alpha' \vee \zeta'') < t + \delta$. By sub-additivity we also have
$$\rh_G(Y, \nu) \geq \rh_G(X, \mu) - \rh_G(X, \mu \given \salg_G(\alpha \vee \zeta)) \geq \rh_G(X, \mu) - \sH(\beta) > t - \delta.$$
Finally, $G \acts (Y, \nu)$ must be a free action since it factors onto $(Z, \eta)$.
\end{proof}

We will now consider the Rokhlin entropy of Bernoulli shifts $(L^G, \lambda^G)$ where $\sH(L, \lambda) < \infty$. Let $\cL$ be the canonical partition of $L^G$. If $\cK$ is a partition coarser than $\cL$, then the translates of $\cK$ are mutually independent and the factor associated to $\salg_G(\cK)$ is a Bernoulli shift $G \acts (K^G, \kappa^G)$. In order to emphasize the fact that $\salg_G(\cK)$ corresponds to a Bernoulli factor of $(L^G, \lambda^G)$, we will write $\cK^G$ for $\salg_G(\cK)$.

\begin{prop} \label{prop:finite}
Let $G$ be a countably infinite group and let $(L, \lambda)$ be a probability space with $L$ finite. Let $\cL$ be the canonical partition of $L^G$ and let $\cK$ be a partition coarser than $\cL$. Then
$$\rh_G \big( L^G, \lambda^G \given \cK^G \big) = \min \Big( \sH(\cL \given \cK), \ \ \suprh{G} \Big).$$
\end{prop}

\begin{proof}
We immediately have $\rh_G(L^G, \lambda^G \given \cK^G) \leq \sH(\cL \given \cK)$ since $\cL$ is a generating partition. We will show that there does not exist any free {\pmp} ergodic action $G \acts (X, \mu)$ with
$$\rh_G(L^G, \lambda^G \given \cK^G) < \rh_G(X, \mu) < \sH(\cL \given \cK).$$
From Lemma \ref{lem:dfactor} it will follow that either $\rh_G(L^G, \lambda^G \given \cK^G) = \sH(\cL \given \cK)$ or else $\rh_G(L^G, \lambda^G \given \cK^G) \geq \rh_G(X, \mu)$ for every free {\pmp} ergodic action $G \acts (X, \mu)$ with $\rh_G(X, \mu) < \infty$.

Towards a contradiction, suppose that $G \acts (X, \mu)$ is a free {\pmp} ergodic action with $\rh_G(L^G, \lambda^G \given \cK^G) < \rh_G(X, \mu) < \sH(\cL \given \cK)$. Fix $\epsilon > 0$ with
$$\rh_G(L^G, \lambda^G \given \cK^G) + \epsilon < \rh_G(X, \mu).$$
By Lemma \ref{lem:mups}, there is an open neighborhood $U \subseteq \E_G(L^G)$ of $\lambda^G$ such that $\rh_G(L^G, \nu \given \cK^G) < \rh_G(L^G, \lambda^G \given \cK^G) + \epsilon / 2$ for all $\nu \in U$. By Corollary \ref{cor:iso}, there is a $G$-equivariant isomorphism $\phi : (X, \mu) \rightarrow (L^G, \nu)$ with $\nu \in U$ and $\rh_{G, \nu}(\cK) < \epsilon / 2$. Then by sub-additivity
\begin{align*}
\rh_G(X, \mu) & = \rh_G(L^G, \nu)\\
 & \leq \rh_{G, \nu}(\cK) + \rh_G(L^G, \nu \given \cK^G)\\
 & < \rh_G(L^G, \lambda^G \given \cK^G) + \epsilon\\
 & < \rh_G(X, \mu),
\end{align*}
a contradiction.
\end{proof}

\begin{thm} \label{thm:finite}
Let $G$ be a countably infinite group and let $(L, \lambda)$ be a probability space with $\sH(L, \lambda) < \infty$. Then
$$\rh_G(L^G, \lambda^G) = \min \Big( \sH(L, \lambda), \ \ \suprh{G} \Big).$$
\end{thm}

\begin{proof}
Let $\cL = \{R_\ell : \ell \in L\}$ be the canonical partition of $L^G$ where
$$R_\ell = \{y \in L^G : y(1_G) = \ell\}.$$
Let $\cL_n$ be an increasing sequence of finite partitions which are coarser than $\cL$ and satisfy $\cL = \bigvee_{n \in \N} \cL_n$. The algebra generated by $\cL_n$ corresponds to a factor $(L_n, \lambda_n)$ of $(L, \lambda)$, and the factor of $(L^G, \lambda^G)$ corresponding to $\cL_n^G$ is $(L_n^G, \lambda_n^G)$. By Lemma \ref{lem:up} $\rh_G(L^G, \lambda^G) = \lim_{n \rightarrow \infty} \rh_G(L_n^G, \lambda_n^G)$. The claim now follows by applying Proposition \ref{prop:finite} to each $(L_n^G, \lambda_n^G)$ and using the fact that $\sH(L_n, \lambda_n) = \sH(\cL_n)$ converges to $\sH(\cL) = \sH(L, \lambda)$.
\end{proof}

\begin{thm} \label{thm:kill}
Let $P$ be a countable group containing arbitrarily large finite subgroups. If $G$ is any countably infinite group with $\suprh{G} < \infty$ then $\suprh{P \times G} = 0$.
\end{thm}

\begin{proof}
Set $\Gamma = P \times G$. Let $(L, \lambda)$ be a probability space with $L$ finite and $\sH(L, \lambda) > 0$, and consider the Bernoulli shift $(L^\Gamma, \lambda^\Gamma)$. By Theorem \ref{thm:finite} it suffices to show that $\rh_\Gamma(L^\Gamma, \lambda^\Gamma) = 0$.

Fix $\epsilon > 0$, fix $k \in \N$ with $\suprh{G} < \log(k)$, and fix a finite subgroup $T \leq P$ with $\log(k) / |T| < \epsilon$. Let $\cL = \{R_\ell : \ell \in L\}$ be the canonical partition of $L^\Gamma$, where
$$R_\ell = \{x \in L^\Gamma : x(1_\Gamma) = \ell\}.$$
Consider the partition $\cL^T$. We may write $\cL^T = \{D_\pi : \pi \in L^T\}$ where
$$D_\pi = \bigcap_{t \in T} t \cdot R_{\pi(t)}.$$
Since $T$ is a group, it naturally acts on $L^T$ by shifts: $(t \cdot \pi)(s) = \pi(t^{-1} s)$. For $u \in T$ we have $u \cdot D_\pi = D_{u \cdot \pi}$ since
$$u \cdot D_\pi = \bigcap_{t \in T} u t \cdot R_{\pi(t)} = \bigcap_{t \in T} t \cdot R_{\pi(u^{-1} t)} = D_{u \cdot \pi}.$$
Let $\cQ = \{Q_{[\pi]} : \pi \in L^T\}$ be the partition of $L^\Gamma$ where $[\pi]$ denotes the $T$-orbit of $\pi$ and
$$Q_{[\pi]} = \bigcup_{t \in T} D_{t \cdot \pi}.$$
Consider the restricted action $G \acts (L^\Gamma, \lambda^\Gamma)$ and let $G \acts (Z, \eta)$ be the factor associated to $\salg_G(\cQ)$. Since $T \cap G = \{1_\Gamma\}$, the $G$-translates of $\cQ$ are mutually independent. As $L^T$ has at least two distinct $T$-orbits, the action $G \acts (Z, \eta)$ is isomorphic to a $G$-Bernoulli shift and is in particular a free action.

By Theorem \ref{thm:robin}, there is a factor $\Gamma \acts (Y, \nu)$ of $(L^\Gamma, \lambda^\Gamma)$ such that $\rh_\Gamma(Y, \nu) < \epsilon$ and the action of $\Gamma$ on $Y$ is free. The $T$-orbits of $Y$ are finite and partition $Y$, so there is a Borel set $M' \subseteq Y$ which meets every $T$-orbit precisely once. Let $\cF$ be the $\Gamma$-invariant sub-$\sigma$-algebra of $L^\Gamma$ associated to $Y$, and let $M \in \cF$ be the pre-image of $M'$.

Define $\xi = \{C_\pi : \pi \in L^T\}$ to be the partition of $L^\Gamma$ defined by
$$C_\pi = \bigcup_{s \in T} s \cdot (D_\pi \cap M).$$
This is indeed a partition of $L^\Gamma$ since the $T$-translates of $M$ partition $L^\Gamma$ and the sets $D_\pi \cap M$ partition $M$. To add clarification to this definition, we remark that $x_1, x_2 \in L^\Gamma$ lie in the same class of $\xi$ if and only if $s_1^{-1} \cdot x_1$ and $s_2^{-1} \cdot x_2$ lie in the same class of $\cL^T$, where $s_1, s_2 \in T$ are defined by the condition $s_1^{-1} \cdot x_1, s_2^{-1} \cdot x_2 \in M$. We observe that $\salg_\Gamma(\xi) \vee \cF = \Borel(L^\Gamma)$ since for $\ell \in L$
\begin{align*}
R_\ell & = \bigcup_{\substack{\pi \in L^T\\\pi(1_\Gamma) = \ell}} D_\pi = \bigcup_{s \in T} \bigcup_{\substack{\pi \in L^T\\\pi(1_\Gamma) = \ell}} \Big( D_\pi \cap s \cdot M \Big) = \bigcup_{s \in T} \bigcup_{\substack{\pi \in L^T\\\pi(1_\Gamma) = \ell}} s \cdot (D_{s^{-1} \cdot \pi} \cap M)\\
 & = \bigcup_{s \in T} \bigcup_{\substack{\pi \in L^T\\\pi(s^{-1}) = \ell}} s \cdot (D_\pi \cap M) = \bigcup_{s \in T} \bigcup_{\substack{\pi \in L^T\\\pi(s^{-1}) = \ell}} \Big( C_\pi \cap s \cdot M \Big).
\end{align*}
Each $C_\pi \in \xi$ is $T$-invariant since for $u \in T$ and $\pi \in L^T$ we have
$$u \cdot C_\pi = \bigcup_{s \in T} (u s) \cdot (D_\pi \cap M) = C_\pi.$$
Furthermore, $\xi$ is finer than $\cQ$ as
\begin{align*}
Q_{[\pi]} & = \bigcup_{t \in T} D_{t \cdot \pi} = \bigcup_{s, t \in T} \Big( D_{t \cdot \pi} \cap s \cdot M \Big) = \bigcup_{s, t \in T} \Big( D_{s t \cdot \pi} \cap s \cdot M \Big)\\
 & = \bigcup_{s, t \in T} s \cdot (D_{t \cdot \pi} \cap M) = \bigcup_{s, t \in T} \Big( C_{t \cdot \pi} \cap s \cdot M \Big) = \bigcup_{t \in T} C_{t \cdot \pi}.
\end{align*}

Let $G \acts (W, \omega)$ be the factor of $(L^\Gamma, \lambda^\Gamma)$ associated to $\salg_G(\xi)$. Since $\xi$ is finer than $\cQ$, $(W, \omega)$ factors onto $(Z, \eta)$. Thus $G$ acts freely on $(W, \omega)$. We have $\rh_G(W, \omega) \leq \sH(\xi) < \infty$ and thus by assumption $\rh_G(W, \omega) \leq \suprh{G} < \log(k)$. Apply Theorem \ref{intro:krieger} to get a $k$-piece generating partition $\beta'$ for $W$, and let $\beta \subseteq \salg_G(\xi)$ be the pre-image of $\beta'$. Then $\xi \subseteq \salg_G(\beta)$ and hence
$$\Borel(L^\Gamma) = \salg_\Gamma(\xi) \vee \cF \subseteq \salg_\Gamma(\beta) \vee \cF.$$
We observed that every $C_\pi \in \xi$ is $T$-invariant. Since $G$ and $T$ commute, it follows that every set in $\salg_G(\xi)$ is $T$-invariant. In particular, each $B \in \beta$ is $T$-invariant. Therefore, setting
$$\alpha = \{L^\Gamma \setminus M\} \cup (\beta \res M),$$
we have $\beta \subseteq \salg_T(\alpha) \vee \cF$. Thus $\Borel(L^\Gamma) = \salg_\Gamma(\alpha) \vee \cF$. Therefore by sub-additivity
\begin{align*}
\rh_\Gamma(L^\Gamma, \lambda^\Gamma) & \leq \rh_\Gamma(Y, \nu) + \rh_\Gamma(L^\Gamma, \lambda^\Gamma \given \cF)\\
 & < \epsilon + \sH(\alpha \given \cF)\\
 & \leq \epsilon + \lambda^\Gamma(M) \cdot \sH_M(\alpha)\\
 & = \epsilon + \frac{1}{|T|} \cdot \sH_M(\beta)\\
 & \leq \epsilon + \frac{1}{|T|} \cdot \log(k)\\
 & < 2 \epsilon.
\end{align*}
Since $\epsilon > 0$ was arbitrary, we conclude that $\rh_\Gamma(L^\Gamma, \lambda^\Gamma) = 0$.
\end{proof}

\section{Rokhlin entropy of Bernoulli shifts: Infinite case} \label{sec:inf}

In this section we study the Rokhlin entropy of $(L^G, \lambda^G)$ when $\sH(L, \lambda) = \infty$. The key idea will be to combine the results of the previous section together with a formula for the Rokhlin entropy of an inverse limit of actions. We remark that there is a strong similarity between the formula we obtain, specifically Corollary \ref{cor:ks}, and the formula for sofic entropy via finite partitions developed by Kerr \cite{Ke13}.

Just as with the continuity properties of the previous section, we mention that the formula for the Rokhlin entropy of inverse limits and its consequences are developed and studied in greater detail in Part III \cite{AS}. The results in Part III will in particular cover actions which are not necessarily ergodic.

\begin{lem} \label{lem:outilim}
Let $G \acts (X, \mu)$ be a {\pmp} ergodic action. Suppose that $G \acts (X, \mu)$ is the inverse limit of actions $G \acts (X_n, \mu_n)$. Identify each $\Borel(X_n)$ as a sub-$\sigma$-algebra of $X$ in the natural way. Let $(\cF_n)_{n \in \N}$ be an increasing sequence of sub-$\sigma$-algebras with $\cF_n \subseteq \Borel(X_n)$ for every $n$, and set $\cF = \bigvee_{n \in \N} \cF_n$. If $\cP$ is a partition with $\cP \subseteq \Borel(X_n)$ for all $n$ and $\inf_{n \in \N} \sH(\cP \given \cF_n) < \infty$ then
$$\rh_{G, \mu}(\cP \given \cF) = \inf_{n \in \N} \rh_{G, \mu_n}(\cP \given \cF_n).$$
\end{lem}

Note that $\mu_n$ appears on the right-hand side of the above expression.

\begin{proof}
If $(X, \mu)$ has an atom then $X_n = X$ and $\cF_n = \cF$ for all sufficiently large $n$. So assume that $(X, \mu)$ is non-atomic. It is immediate from the definitions that $\rh_{G, \mu}(\cP \given \cF) \leq \inf_n \rh_{G, \mu_n}(\cP \given \cF_n)$. So we only need to consider the reverse inequality.

Note that $\rh_{G, \mu}(\cP \given \cF) \leq \sH_\mu(\cP \given \cF) < \infty$. Fix $\delta > 0$ and fix a countable partition $\xi'$ with $\sH_\mu(\xi' \given \cF) < \rh_{G, \mu}(\cP \given \cF) + \delta$ and $\cP \subseteq \salg_G(\xi') \vee \cF$. By Theorem \ref{thm:relbasic} there is a partition $\xi$ with $\sH_\mu(\xi) < \rh_{G, \mu}(\cP \given \cF) + \delta$ and $\cP \subseteq \salg_G(\xi) \vee \cF$. Since $\inf_k \sH_\mu(\cP \given \cF_k) < \infty$, by Lemma \ref{lem:shan} there are finite $T \subseteq G$ and $k \in \N$ such that
$$\sH_\mu(\cP \given \xi^T \vee \cF_k) < \delta.$$
Using the dense algebra $\bigcup_n \Borel(X_n)$, apply Lemma \ref{lem:algebra} to obtain $n \geq k$ and $\beta \subseteq \Borel(X_n)$ with $\dR_\mu(\beta, \xi) < \delta / (2 |T|)$. Then we have
\begin{align*}
\sH_\mu(\cP \given \salg_G(\beta) \vee \cF_n) & \leq \sH_\mu(\cP \given \beta^T \vee \cF_k)\\
 & \leq \sH_\mu(\cP \given \xi^T \vee \cF_k) + 2|T| \cdot \dR_\mu(\beta, \xi) < 2 \delta.
\end{align*}
Since $\cP, \beta \subseteq \Borel(X_n)$, the partition $\beta$ naturally corresponds to a partition of $X_n$ and the above inequality becomes
$$\sH_{\mu_n}(\cP \given \salg_G(\beta) \vee \cF_n) < 2 \delta.$$
Therefore by sub-additivity
\begin{align*}
\rh_{G, \mu_n}(\cP \given \cF_n) & \leq \rh_{G, \mu_n}(\beta \given \cF_n) + \rh_{G, \mu_n}(\cP \given \salg_G(\beta) \vee \cF_n)\\
 & \leq \sH_{\mu_n}(\beta) + \sH_{\mu_n}(\cP \given \salg_G(\beta) \vee \cF_n)\\
 & < \sH_{\mu_n}(\beta) + 2 \delta\\
 & < \sH_\mu(\xi) + 3 \delta\\
 & < \rh_{G, \mu}(\cP \given \cF) + 4 \delta.
\end{align*}
Now take the infimum over $n \in \N$ and let $\delta$ tend to $0$.
\end{proof}

\begin{cor} \label{cor:relup}
Let $G \acts (X, \mu)$ be a {\pmp} ergodic action, and let $(\cF_n)_{n \in \N}$ be an increasing sequence of sub-$\sigma$-algebras. Set $\cF = \bigvee_{n \in \N} \cF_n$.
\begin{enumerate}
\item[\rm (i)] $\rh_{G, \mu}(\cP \given \cF) = \inf_{n \in \N} \rh_{G, \mu}(\cP \given \cF_n)$ if $\cP$ is a partition with $\inf_n \sH(\cP \given \cF_n) < \infty$.
\item[\rm (ii)] $\rh_G(X, \mu \given \cF) = \inf_{n \in \N} \rh_G(X, \mu \given \cF_n)$ if the right-hand side is finite.
\end{enumerate}
\end{cor}

\begin{proof}
(i). This is immediate from Lemma \ref{lem:outilim} by taking each $(X_n, \mu_n) = (X, \mu)$.

(ii). Assume that the right-hand side is finite. Fix $k \in \N$ with $\rh_G(X, \mu \given \cF_k) < \infty$. Fix a partition $\cP$ satisfying $\sH(\cP \given \cF_k) < \infty$ and $\salg_G(\cP) \vee \cF_k = \Borel(X)$. Then $\salg_G(\cP) \vee \cF_n = \Borel(X)$ for all $n \geq k$, hence
$$\rh_G(X, \mu \given \cF) = \rh_{G, \mu}(\cP \given \cF) \quad \text{and} \quad \rh_G(X, \mu \given \cF_n) = \rh_{G, \mu}(\cP \given \cF_n) \text{ for } n \geq k.$$
Since $\inf_n \sH(\cP \given \cF_n) < \infty$, we can apply (i) to obtain
\begin{equation*}
\rh_G(X, \mu \given \cF) = \rh_{G, \mu}(\cP \given \cF) = \inf_{n \geq k} \rh_{G, \mu}(\cP \given \cF_n) = \inf_{n \geq k} \rh_G(X, \mu \given \cF_n).\qedhere
\end{equation*}
\end{proof}

Now we state the formula for the Rokhlin entropy of inverse limits.

\begin{thm} \label{thm:ks}
Let $G \acts (X, \mu)$ be a {\pmp} ergodic action and let $\cF$ be a $G$-invariant sub-$\sigma$-algebra. Suppose that $G \acts (X, \mu)$ is the inverse limit of actions $G \acts (X_n, \mu_n)$. Identify each $\Borel(X_n)$ as a sub-$\sigma$-algebra of $X$ in the natural way. Then
\begin{equation} \label{eqn:ks0}
\rh_G(X, \mu \given \cF) < \infty \Longleftrightarrow
\left\{\begin{array}{c}
\displaystyle{\inf_{n \in \N} \sup_{m \geq n} \rh_{G, \mu}(\Borel(X_m) \given \Borel(X_n) \vee \cF) = 0}\\
\displaystyle{\text{and} \quad \forall m \ \rh_{G, \mu}(\Borel(X_m) \given \cF) < \infty.}
\end{array}\right\}\end{equation}
Furthermore, when $\rh_G(X, \mu \given \cF) < \infty$ we have
\begin{equation} \label{eqn:ks}
\rh_G(X, \mu \given \cF) = \sup_{m \in \N} \rh_{G, \mu}(\Borel(X_m) \given \cF).
\end{equation}
\end{thm}

We remark that we do not know if (\ref{eqn:ks}) is true in general without assuming $\rh_G(X, \mu \given \cF) < \infty$.

\begin{proof}
First suppose that $\rh_G(X, \mu \given \cF) < \infty$. Then
$$\rh_{G, \mu}(\Borel(X_m) \given \cF) \leq \rh_G(X, \mu \given \cF) < \infty$$
for all $m \in \N$ and by applying Corollary \ref{cor:relup}.(ii) we get
\begin{align*}
0 = \rh_G(X, \mu \given \Borel(X)) & = \inf_{n \in \N} \rh_G(X, \mu \given \Borel(X_n) \vee \cF)\\
 & \geq \inf_{n \in \N} \sup_{m \geq n} \rh_{G, \mu}(\Borel(X_m) \given \Borel(X_n) \vee \cF) \geq 0.
\end{align*}
This proves one implication in the first claim.

Now suppose that the right-side of (\ref{eqn:ks0}) is true. For each $i \geq 1$ fix $n(i)$ with
$$\sup_{m \in \N} \rh_{G, \mu}(\Borel(X_m) \given \Borel(X_{n(i)}) \vee \cF) < \frac{\delta}{2^i}.$$
Then by using $m = n(i+1)$ we have
$$\rh_{G, \mu}(\Borel(X_{n(i+1)}) \given \Borel(X_{n(i)}) \vee \cF) < \frac{\delta}{2^i}.$$
Now by sub-additivity we have
\begin{align*}
\rh_G(X, \mu \given \cF) & \leq \rh_{G, \mu}(\Borel(X_{n(1)}) \given \cF) + \sum_{i=1}^\infty \rh_{G, \mu}(\Borel(X_{n(i+1)}) \given \Borel(X_{n(i)}) \vee \cF)\\
 & < \rh_{G, \mu}(\Borel(X_{n(1)}) \given \cF) + \delta.
\end{align*}
So $\rh_G(X, \mu \given \cF) < \infty$, completing the proof of the first claim. The second claim also follows, since above we only assumed that the right-side of (\ref{eqn:ks0}) was true (equivalently $\rh_G(X, \mu \given \cF) < \infty$ by the first claim). By letting $\delta$ tend to $0$ above, we get that $\rh_G(X, \mu \given \cF) \leq \sup_m \rh_{G, \mu}(\Borel(X_m) \given \cF)$. The reverse inequality is immediate from the definitions.
\end{proof}

Notice the the formula in the previous theorem relies upon outer Rokhlin entropies computed within the largest space $X$. When expressing $G \acts (X, \mu)$ as an inverse limit, it may be more natural to express the Rokhlin entropy of $(X, \mu)$ purely in terms of the actions which build the inverse limit. In order to express Rokhlin entropy in this way, we must assume that each $\rh_G(X_n, \mu_n \given \cF_n)$ is finite.

\begin{cor} \label{cor:ks}
Let $G \acts (X, \mu)$ be a {\pmp} ergodic action. Suppose that $G \acts (X, \mu)$ is the inverse limit of actions $G \acts (X_n, \mu_n)$. Identify each $\Borel(X_n)$ as a sub-$\sigma$-algebra of $X$ in the natural way. Let $(\cF_n)_{n \in \N}$ be an increasing sequence of sub-$\sigma$-algebras with $\cF_n \subseteq \Borel(X_n)$ for every $n$, and set $\cF = \bigvee_{n \in \N} \cF_n$. Assume that $\rh_G(X_n, \mu_n \given \cF_n) < \infty$ for all $n$. Then
\begin{equation*}
\rh_G(X, \mu \given \cF) < \infty \Longleftrightarrow \inf_{n \in \N} \sup_{m \geq n} \inf_{k \geq m} \rh_{G, \mu_k}(\Borel(X_m) \given \Borel(X_n) \vee \cF_k) = 0.
\end{equation*}
Furthermore, when $\rh_G(X, \mu \given \cF) < \infty$ we have
\begin{equation*}
\rh_G(X, \mu \given \cF) = \sup_{m \in \N} \inf_{k \geq m} \rh_{G, \mu_k}(\Borel(X_m) \given \cF_k).
\end{equation*}
\end{cor}

\begin{proof}
For each $m$ pick a partition $\alpha_m \subseteq \Borel(X_m)$ with $\sH(\alpha_m \given \cF_m) < \infty$ and $\Borel(X_m) = \salg_G(\alpha_m) \vee \cF_m$. Then by Lemma \ref{lem:outilim} we have
$$\rh_{G, \mu}(\Borel(X_m) \given \cF) = \rh_{G, \mu}(\alpha_m \given \cF) = \inf_{k \geq m} \rh_{G, \mu_k}(\alpha_m \given \cF_k) = \inf_{k \geq m} \rh_{G, \mu_k}(\Borel(X_m) \given \cF_k)$$
and by the same reasoning for every $n \leq m$
$$\rh_{G, \mu}(\Borel(X_m) \given \Borel(X_n) \vee \cF) = \inf_{k \geq m} \rh_{G, \mu_k}(\Borel(X_m) \given \Borel(X_n) \vee \cF_k).$$
So the corollary follows from the two identities above and Theorem \ref{thm:ks}.
\end{proof}

Now we proceed to consider Bernoulli shifts $(L^G, \lambda^G)$ with $\sH(L, \lambda) = \infty$. First we need a lemma.

\begin{lem} \label{lem:carve}
Let $(L, \lambda)$ be a probability space with $\sH(L, \lambda) = \infty$, and let $c > 0$. Then there exists a sequence of finite partitions $(\cL_n)_{n \in \N}$ with $\bigvee_{n \in \N} \salg(\cL_n) = \Borel(L)$ and
$$\sH \Big( \textstyle{ \cL_m \Given \bigvee_{n \neq m} \cL_n } \Big) > c$$
for all $m \in \N$.
\end{lem}

\begin{proof}
First suppose that $L$ is essentially countable. For $\ell \in L$ we will write $\lambda(\ell)$ for $\lambda(\{\ell\})$. Since
$$\sum_{\ell \in L} - \lambda(\ell) \cdot \log \lambda(\ell) = \sH(L, \lambda) = \infty,$$
we can partition $L$ into finite sets $I_n$ with
$$\sum_{\ell \in I_n} - \lambda(\ell) \cdot \log \lambda(\ell) > c + \log(2)$$
for all $n$. Define
$$\cL_n = \{L \setminus I_n\} \cup \Big\{ \{\ell\} : \ell \in I_n \Big\}.$$
Note that $\sH(\cL_n) > c + \log(2)$. Clearly $\cL_n$ is finite and $\bigvee_{n \in \N} \salg(\cL_n) = \Borel(L)$. Additionally, we have $I_n \in \bigvee_{k \neq n} \cL_k$ since $L \setminus I_n$ is the union of all singleton sets contained in $\bigvee_{k \neq n} \cL_k$. Therefore
\begin{align*}
\sH \Big( \textstyle{\cL_n \given \bigvee_{k \neq n} \cL_k} \Big) & = \sH(\cL_n \given \{I_n, L \setminus I_n\})\\
 & = \sH(\cL_n) - \sH(\{I_n, L \setminus I_n\})\\
 & \geq \sH(\cL_n) - \log(2)\\
 & > c.
\end{align*}

Now suppose that $(L, \lambda)$ is not essentially countable. Then $L$ decomposes into a non-atomic part $B \subseteq L$ and a purely atomic part $A \subseteq L$ with $\{B, A\}$ a partition of $L$ and $\lambda(B) > 0$. Fix any increasing sequence $\alpha_n$ of finite partitions of $A$ with $\Borel(L) \res A = \bigvee_{n \in \N} \salg(\alpha_n) \res A$. Choose a probability vector $\pv$ with $\mu(B) \cdot \sH(\pv) > c$, and let $\lambda_B$ be the normalized restriction of $\lambda$ to $B$. Since $B$ has no atoms, we can find a sequence of $\lambda_B$-independent ordered partitions $\beta_n$ of $B$ with $\dist_{\lambda_B}(\beta_n) = \pv$ for every $n$ and with $\Borel(L) \res B = \bigvee_{n \in \N} \salg(\beta_n) \res B$. Now set $\cL_n = \beta_n \cup \alpha_n$. Then $\cL_n$ is finite and $\Borel(L) = \bigvee_{n \in \N} \salg(\cL_n)$. Finally, since $\{B, A\}$ is coarser than every $\cL_n$ we have
\begin{align*}
\sH \Big( \textstyle{\cL_m \given \bigvee_{n \neq m} \cL_n} \Big) & \geq \lambda(B) \cdot \sH_B \Big( \textstyle{\cL_m \given \bigvee_{n \neq m} \cL_n} \Big)\\
 & = \lambda(B) \cdot \sH_B \Big( \textstyle{\beta_m \given \bigvee_{n \neq m} \beta_n} \Big)\\
 & = \lambda(B) \cdot \sH(\pv)\\
 & > c.\qedhere
\end{align*} 
\end{proof}

\begin{thm} \label{thm:infinite}
Let $G$ be a countably infinite group, and let $(L, \lambda)$ be a probability space with $\sH(L, \lambda) = \infty$. Then $\rh_G(L^G, \lambda^G) = \infty$ if and only if there is a free ergodic {\pmp} action $G \acts (X, \mu)$ with $\rh_G(X, \mu) > 0$.
\end{thm}

\begin{proof}
One implication is immediate: if $\rh_G(L^G, \lambda^G) = \infty$ then in particular $\rh_G(L^G, \lambda^G) > 0$. So suppose that $G \acts (X, \mu)$ is a free {\pmp} ergodic action with $\rh_G(X, \mu) > 0$. Let $(\alpha_n)$ be an increasing sequence of finite partitions of $X$ with $\Borel(X) = \bigvee_{n \in \N} \salg_G(\alpha_n)$. For each $n$ let $G \acts (X_n, \mu_n)$ be the factor of $(X, \mu)$ associated to $\salg_G(\alpha_n)$. Using Theorem \ref{thm:robin}, we may choose $\alpha_1$ so that $G$ acts freely on every $(X_n, \mu_n)$. By Theorem \ref{thm:ks}, there must be $n \in \N$ with $\rh_G(X_n, \mu_n) > 0$. Since also $\rh_G(X_n, \mu_n) \leq \sH(\alpha_n) < \infty$, we conclude that $\suprh{G} > 0$. Fix $c \in \R$ with $0 < c \leq \suprh{G}$.

Apply Lemma \ref{lem:carve} to get a sequence $\cL_n$ of finite non-trivial partitions of $L$ with $\Borel(L) = \bigvee_{n \in \N} \salg(\cL_n)$ and $\sH(\cL_m \given \bigvee_{n \neq m} \cL_n) \geq c$ for all $m$. For $m \leq k$ set
$$\cL_{[0,k]} = \bigvee_{0 \leq i \leq k} \cL_i \quad \text{and} \quad \cL_{[0,k],m} = \bigvee_{0 \leq i \neq m \leq k} \cL_i.$$
Note that for $k \geq m$ we have $\sH(\cL_{[0,k]} \given \cL_{[0,k],m}) \geq c$ by construction. We let $(L_{[0,k]}, \lambda_{[0,k]})$ denote the factor of $(L, \lambda)$ associated to $\cL_{[0,k]}$. Let $\cL = \{R_\ell : \ell \in L\}$ be the canonical (possibly uncountable) partition of $L^G$ defined by
$$R_\ell = \{w \in L^G : w(1_G) = \ell\}.$$
Note that $\Borel(L^G) = \cL^G$. We identify each of the partitions $\cL_m$, $\cL_{[0,k]}$, and $\cL_{[0,k],m}$ as coarsenings of $\cL \subseteq \Borel(L^G)$. Note that $(L_{[0,k]}^G, \lambda_{[0,k]}^G)$ is the factor of $(L^G, \lambda^G)$ associated to $\cL_{[0,k]}^G$. For all $n < m \leq k$ we have
$$\rh_{G, \lambda_{[0,k]}^G}(\cL_{[0,m]} \given \cL_{[0,n]}^G) \geq \rh_{G, \lambda_{[0,k]}^G}(\cL_m \given \cL_{[0,k],m}^G) = \rh_G(L_{[0,k]}^G, \lambda_{[0,k]}^G \given \cL_{[0,k],m}^G).$$
So by Proposition \ref{prop:finite} we have
\begin{align*}
\inf_{n \in \N} \sup_{m \geq n} \inf_{k \geq m} \rh_{G, \lambda_{[0,k]}^G}(\cL_{[0,m]}^G \given \cL_{[0,n]}^G) & \geq \sup_{m \in \N} \inf_{k \geq m} \rh_G(L_{[0,k]}^G, \lambda_{[0,k]}^G \given \cL_{[0,k],m}^G)\\
& = \sup_{m \in \N} \inf_{k \geq m} \min \Big( \sH(\cL_{[0,k]} \given \cL_{[0,k],m}), \ \suprh{G} \Big)\\
& \geq c > 0.
\end{align*}
Therefore $\rh_G(L^G, \lambda^G) = \infty$ by Corollary \ref{cor:ks}.
\end{proof}

\begin{cor} \label{cor:infjump}
Let $G$ be a countably infinite group. The following are equivalent:
\begin{enumerate}
\item[\rm (i)] $\suprh{G} > 0$;
\item[\rm (ii)] there is a free ergodic {\pmp} action with $0 < \rh_G(X, \mu) < \infty$;
\item[\rm (iii)] there is a free ergodic {\pmp} action with $\rh_G(X, \mu) = \infty$.
\end{enumerate}
\end{cor}

\begin{proof}
The equivalence of (i) and (ii) is by definition. Theorem \ref{thm:infinite} shows that (ii) implies (iii), and the implication (iii) implies (ii) was deduced in the first paragraph of the proof of Theorem \ref{thm:infinite}.
\end{proof}

We mention that if in Theorem \ref{thm:ks} the equation (\ref{eqn:ks}) holds without assuming $\rh_G(X, \mu \given \cF) < \infty$, then from a free ergodic action $G \acts (Y, \nu)$ with $\rh_G(Y, \nu) = \infty$ one could use the argument in the first paragraph of the proof of Theorem \ref{thm:infinite} to show that $(Y, \nu)$ has free factors with arbitrarily large but finite Rokhlin entropy values. From Corollary \ref{cor:infjump} it would then follow that $\suprh{G} > 0$ implies $\suprh{G} = \infty$.

\begin{cor}
Assume that every countably infinite group $G$ admits a free ergodic {\pmp} action with $\rh_G(X, \mu) > 0$. Then:
\begin{enumerate}
\item[\rm (i)] $\rh_G(L^G, \lambda^G) = \sH(L, \lambda)$ for all countably infinite groups $G$ and all probability spaces $(L, \lambda)$.
\item[\rm (ii)] All Bernoulli shifts over countably infinite groups have completely positive outer Rokhlin entropy.
\item[\rm (iii)] Gottschalk's surjunctivity conjecture and Kaplansky's direct finiteness conjecture are true.
\end{enumerate}
\end{cor}

\begin{proof}
It follows from Corollary \ref{cor:infjump} and Theorem \ref{thm:kill} that $\suprh{G} = \infty$ for all countably infinite groups $G$. By applying Theorems \ref{thm:finite} and \ref{thm:infinite} we obtain (i). From Corollaries \ref{cor:gott} and \ref{cor:cpe} we obtain (ii) and (iii).
\end{proof}

\appendix
\section{Metrics on the space of partitions} \label{sec:metric}

Let $(X, \mu)$ be a probability space. Recall that the \emph{measure algebra} of $(X, \mu)$ is the algebra of equivalence classes of Borel sets mod null sets together with the metric $\dB_\mu(A, B) = \mu(A \symd B)$. There is a closely related metric $\dB_\mu$ on the space of all countable Borel partitions $\Prt$ defined by
$$\dB_\mu(\alpha, \beta) = \inf \Big\{ \mu(Y) \: Y \subseteq X \text{ and } \alpha \res (X \setminus Y) = \beta \res (X \setminus Y) \Big\}.$$
We will tend to work more frequently with the space $\HPrt$ of countable Borel partitions $\alpha$ satisfying $\sH(\alpha) < \infty$. In addition to the metric $\dB_\mu$, this space also has the \emph{Rokhlin metric} $\dR_\mu$ defined by
$$\dR_\mu(\alpha, \beta) = \sH(\alpha \given \beta) + \sH(\beta \given \alpha).$$

\begin{lem} \label{lem:dbarr}
Let $G$ be a countable group, let $G \acts (X, \mu)$ be a {\pmp} action, let $\cF$ be a $G$-invariant sub-$\sigma$-algebra, and let $\alpha, \beta, \xi \in \HPrt$. Then:
\begin{enumerate}
\item[\rm (i)] $\dR_\mu(\beta^T, \xi^T) \leq |T| \cdot \dR_\mu(\beta, \xi)$ for every finite $T \subseteq G$;
\item[\rm (ii)] $\dR_\mu(\alpha \vee \beta, \alpha \vee \xi) \leq \dR_\mu(\beta, \xi)$;
\item[\rm (iii)] $|\sH(\beta \given \cF) - \sH(\xi \given \cF)| \leq \dR_\mu(\beta, \xi)$;
\item[\rm (iv)] $|\sH(\alpha \given \beta \vee \cF) - \sH(\alpha \given \xi \vee \cF)| \leq 2 \cdot \dR_\mu(\beta, \xi)$.
\end{enumerate}
\end{lem}

\begin{proof}
We have
$$\sH(\beta^T \given \xi^T) \leq \sum_{t \in T} \sH(t \cdot \beta \given \xi^T) \leq \sum_{t \in T} \sH(t \cdot \beta \given t \cdot \xi) = |T| \cdot \sH(\beta \given \xi),$$
where the final equality holds since $G$ acts measure-preservingly. This establishes (i). Item (ii) is immediate since $\sH(\alpha \vee \beta \given \alpha \vee \xi) = \sH(\beta \given \alpha \vee \xi) \leq \sH(\beta \given \xi)$. For (iii), we may assume that $\sH(\beta \given \cF) \geq \sH(\xi \given \cF)$. Then we have
$$\sH(\beta \given \cF) - \sH(\xi \given \cF) \leq \sH(\beta \vee \xi \given \cF) - \sH(\xi \given \cF) = \sH(\beta \given \xi \vee \cF) \leq \sH(\beta \given \xi) \leq \dR_\mu(\beta, \xi).$$
Item (iv) follows from (ii) and (iii) by using the identity $\sH(\alpha \given \beta \vee \cF) = \sH(\alpha \vee \beta \given \cF) - \sH(\beta \given \cF)$.
\end{proof}

In the next lemma we will use the well-known property \cite[Fact 1.7.7]{Do11} that for every $n \in \N$, the restrictions of $\dB_\mu$ and $\dR_\mu$ to the space of $n$-piece partitions are uniformly equivalent. Moreover, $\dB_\mu$ is always uniformly dominated by $\dR_\mu$, meaning that for every $\epsilon > 0$ there is $\delta > 0$ such that if $\alpha, \beta \in \HPrt$ and $\dR_\mu(\alpha, \beta) < \delta$ then $\dB_\mu(\alpha, \beta) < \epsilon$.

\begin{lem} \label{lem:coarse}
Let $G \acts (X, \mu)$ be a {\pmp} action. Let $T \subseteq G$ be finite, let $\alpha \in \HPrt$, and let $\beta$ be a coarsening of $\alpha^T$. For every $\epsilon > 0$ there is $\delta > 0$ so that if $\alpha' \in \HPrt$ and $\dR_\mu(\alpha', \alpha) < \delta$, then there is a coarsening $\beta'$ of $\alpha'^T$ with $\dR_\mu(\beta', \beta) < \epsilon$.
\end{lem}

\begin{proof}
By Lemma \ref{lem:shan}, there is a finite partition $\beta_0$ coarser than $\beta$ with $\dR_\mu(\beta_0, \beta) < \epsilon / 2$. Set $n = |\beta_0|$ and let $\kappa > 0$ be such that $\dR_\mu(\zeta, \zeta') < \epsilon / 2$ whenever $\zeta$ and $\zeta'$ are $n$-piece partitions with $\dB_\mu(\zeta, \zeta') < \kappa$. Let $\delta > 0$ be such that $\dB_\mu(\xi, \xi') < \kappa / |T|$ whenever $\xi, \xi' \in \HPrt$ satisfy $\dR_\mu(\xi, \xi') < \delta$. Now let $\alpha' \in \HPrt$ with $\dR_\mu(\alpha', \alpha) < \delta$. Then $\dB_\mu(\alpha', \alpha) < \kappa / |T|$ and hence $\dB_\mu(\alpha'^T, \alpha^T) < \kappa$. This means there is a set $Y \subseteq X$ with $\mu(Y) < \kappa$ and $\alpha'^T \res (X \setminus Y) = \alpha^T \res (X \setminus Y)$. Thus there is a $n$-piece coarsening $\beta'$ of $\alpha'^T$ with $\beta' \res (X \setminus Y) = \beta_0 \res (X \setminus Y)$. So $\dB_\mu(\beta', \beta_0) < \kappa$ and hence $\dR_\mu(\beta', \beta_0) < \epsilon / 2$. We conclude that $\dR_\mu(\beta', \beta) < \epsilon$.
\end{proof}

\begin{lem} \label{lem:algebra}
Let $(X, \mu)$ be a probability space, and let $(\mathcal{A}_n)_{n \in \N}$ be an increasing sequence of algebras of Borel sets whose union is $\dB_\mu$-dense in a sub-$\sigma$-algebra $\cF$. If $\beta \in \HPrt$, $\beta \subseteq \cF$, and $\epsilon > 0$ then there is $k \in \N$ and a partition $\beta' \subseteq \mathcal{A}_k$ with $\dR_\mu(\beta', \beta) < \epsilon$.
\end{lem}

\begin{proof}
By Lemma \ref{lem:shan} there is a finite partition $\beta_0$ coarser than $\beta$ with $\dR_\mu(\beta_0, \beta) < \epsilon / 2$. Set $n = |\beta_0|$ and let $\delta > 0$ be such that $\dR_\mu(\zeta, \zeta') < \epsilon / 2$ whenever $\zeta$ and $\zeta'$ are $n$-piece partitions with $\dB_\mu(\zeta, \zeta') < \delta$. Since the $\mathcal{A}_k$'s are increasing and have dense union in $\cF$ and since $\beta_0$ is finite, there is $k \in \N$ and a $n$-piece partition $\beta' \subseteq \mathcal{A}_k$ with $\dB_\mu(\beta', \beta_0) < \delta$. Then $\dR_\mu(\beta', \beta_0) < \epsilon / 2$ and $\dR_\mu(\beta', \beta) < \epsilon$.
\end{proof}

\begin{cor} \label{cor:gen}
Let $G \acts (X, \mu)$ be a {\pmp} action, let $\cF$ be a sub-$\sigma$-algebra, and let $\alpha$ be a partition with $\cF \subseteq \salg_G(\alpha)$. If $\beta \in \HPrt$, $\beta \subseteq \cF$, and $\epsilon > 0$, then there exists a finite $T \subseteq G$ and a coarsening $\beta'$ of $\alpha^T$ with $\dR_\mu(\beta', \beta) < \epsilon$.
\end{cor}

\begin{proof}
Pick an increasing sequence of finite sets $(T_n)_{n \in \N}$ with $\bigcup_n T_n = G$. Let $\mathcal{A}_n$ be the algebra generated by $\alpha^{T_n}$. Then $\bigcup_n \mathcal{A}_n$ is dense in $\cF$ since $\cF \subseteq \salg_G(\alpha)$. Now apply Lemma \ref{lem:algebra}.
\end{proof}

The same proof also provides the following.

\begin{cor} \label{cor:invgen}
Let $G \acts (X, \mu)$ be a {\pmp} action, let $\cF$ be a sub-$\sigma$-algebra, and let $(\alpha_n)$ be an increasing sequence of partitions with $\cF \subseteq \bigvee_{n \in \N} \salg_G(\alpha_n)$. If $\beta \in \HPrt$, $\beta \subseteq \cF$, and $\epsilon > 0$, then there exist $k \in \N$, a finite $T \subseteq G$, and a coarsening $\beta'$ of $\alpha_k^T$ with $\dR_\mu(\beta', \beta) < \epsilon$.
\end{cor}

\end{document}